\documentclass[12pt]{amsart}
\usepackage{latexsym,amssymb,amsmath,amsfonts}
\usepackage{latexsym,amssymb,amsthm,tikz,verbatim,cancel,hyperref,dsfont}
\usepackage{multicol}
\usepackage[normalem]{ulem}

\numberwithin{equation}{section}
\newtheorem{theorem}{Theorem}[section]
\newtheorem{lemma}[theorem]{Lemma}
\newtheorem{proposition}[theorem]{Proposition}
\newtheorem{corollary}[theorem]{Corollary}

\theoremstyle{definition}
\newtheorem{definition}[theorem]{Definition}

\newtheorem{remark}[theorem]{Remark}
\newtheorem{example}[theorem]{Example}

\begin{document}

\title{Hadamard Products and Binomial Ideals}
\thanks{Version: November 24, 2022}

\author[B.~Atar et al.]{B\"u\c{s}ra Atar}
\author[]{Kieran Bhaskara}
\author[]{Adrian Cook}
\author[]{Sergio Da Silva}
\author[]{Megumi Harada}
\author[]{Jenna Rajchgot}
\author[]{Adam Van Tuyl}
\author[]{Runyue Wang} 
\author[]{Jay Yang}
\address{Department of Mathematics \& Statistics, McMaster University, Hamilton, ON L8S 4L8, Canada}
\email{atarb@mcmaster.ca,~kieran.bhaskara@mcmaster.ca,~cooka11@mcmaster.ca,\newline Megumi.Harada@math.mcmaster.ca,~rajchgoj@mcmaster.ca,~vantuyl@math.mcmaster.ca,\newline wangr109@mcmaster.ca}
\address{Department of Mathematics and Economics, Virginia State University, Petersburg, VA 23806, USA}
\email{sdasilva@vsu.edu}
\address{Department of Mathematics, Washington University at St. Louis, St. Louis, MO 63130, USA}
\email{jayy@wustl.edu}

\date{\today}

\begin{abstract}
We study the Hadamard product of 
two varieties $V$ and $W$, with particular attention to the situation when one or both of $V$ and $W$ is a  
binomial variety.  The main result of this paper shows that when $V$ and $W$ are both binomial varieties, and the binomials
that define $V$ and $W$ have the same binomial
exponents, then the defining equations of $V \star W$ can be computed explicitly and directly from the defining equations of
$V$ and $W$.  This result recovers known 
results about Hadamard products of binomial hypersurfaces and toric varieties. 
Moreover, as an application of our main result,
we describe a relationship between
the Hadamard product of the toric
ideal $I_G$ of a graph $G$ and the toric ideal $I_H$
of a subgraph $H$ of $G$.
We also derive results about algebraic invariants of Hadamard products: assuming $V$ and $W$ are binomial with the same exponents, we show that $\deg(V\star W) = \deg(V)=\deg(W)$ and $\dim(V\star W) = \dim(V)=\dim(W)$. 
Finally, given any (not necessarily binomial) projective variety $V$ and a point $p \in \mathbb{P}^n \setminus \mathbb{V}(x_0x_1\cdots x_n)$,
subject to some additional minor hypotheses, we find an explicit binomial variety that describes all the points $q$
that satisfy $p \star V = q\star V$.
\end{abstract}


\maketitle

\section{Introduction}

Let $V$ and $W$ be projective varieties in $\mathbb{P}^n$
over an algebraically closed field $k$ of characteristic zero.   The last decade
has seen increasing interest in, and development of, the theory of the \emph{Hadamard product} 
$V \star W$ of $V$ and $W$.  
Informally (for formal definitions see Section~\ref{sec:background}), the
Hadamard product $V \star W$ is the Zariski closure of
the rational map $\varphi:V \times W \rightarrow \mathbb{P}^n$
given by
\begin{equation}\label{map}
([p_0:\cdots:p_n],[q_0:\cdots:q_n]) \mapsto [p_0q_0:\cdots:
p_nq_n].
\end{equation}
The notion of a Hadamard product of varieties first appeared
in  the work of Cueto, Morton, and Sturmfels \cite{CMS} and
Cueto, Tobis, and Yu \cite{CTY}, where the Hadamard product appears in their study of Boltzman
machines and statistical models.  Hadamard products have also  
appeared in the theory of tropical geometry (cf. \cite[Proposition 5.5.11]{MS}).

A systematic study of Hadamard products was initiated in 2016 by Bocci, Carlini, and Kileel \cite{BCK} from the point of view of algebraic
geometry and commutative algebra.  In the years following, several groups of researchers 
\cite{BJBGM,BCFL,BCFL2,CCFL,FOW} have continued this line of investigation.  In another direction, some recent research has focused instead on the question of determining classes of algebraic varieties -- such as star configurations -- which can be realized as the Hadamard product of two varieties \cite{BJC,BC,BCC,CCGV,CCFG}.

In this paper we focus primarily on the study of the Hadamard product $V \star W$ under the extra hypothesis that one or both of 
$V$ and $W$ are binomial varieties; that is, the defining 
ideals $\mathbb{I}(V)$ and $\mathbb{I}(W)$ are generated by binomials.\footnote{In this paper,
we define a binomial to be a polynomial with \emph{exactly} two terms (see Remark
\ref{binomialremark} for more details).} Binomial ideals
are ubiquitous in algebraic geometry and commutative algebra;  toric 
geometry and algebraic statistics are just two examples of research areas in which they 
appear.   We point the readers to Herzog, Hibi and Ohsugi's text 
\cite{HHO} and Eisenbud and Sturmfels's 
influential paper \cite{ES} for more history and context on the topic of binomial ideals and binomial varieties. Given the importance of binomial varieties, it is of interest 
to understand how the Hadamard product behaves with respect to this family.  
Some results in this direction are already known.  For example,
Friedenberg, Oneto, and Williams
\cite{FOW} have shown that the Hadamard product $\star$ is idempotent when restricted to the class of toric varieties, i.e., if $V$ is a toric variety (and hence a binomial variety), then $V \star V = V$. Moreover, Bocci
and Carlini \cite{BC} have recently shown that, under certain hypotheses, if  
$V$ and $W$ are binomial hypersurfaces, then $V \star W$ is also
a binomial hypersurface. The results of this manuscript will recover both of the above statements as corollaries. 

Our main contribution, Theorem~\ref{maintheorem} below,  significantly simplifies the computation of the Hadamard product for certain pairs of binomial varieties, as we now describe.  We say that two
binomials, written as $ax^{\alpha_0}_0x_1^{\alpha_1}\cdots x_n^{\alpha_n} -
bx_0^{\beta_0}x_1^{\beta_1}\cdots x_n^{\beta_n}$ and
$cx_0^{\gamma_0}x_1^{\gamma_1} \cdots x_n^{\gamma_n} - dx_0^{\delta_0}
x_1^{\delta_1}\cdots x_n^{\delta_n}$, have the \emph{same binomial
exponents} if 
$$(\alpha_0,\ldots,\alpha_n) = (\gamma_0,\ldots,\gamma_n)
~~\mbox{and}~~ (\beta_0,\ldots,\beta_n) = (\delta_0,\ldots,\delta_n).$$
We additionally say that binomial varieties $V$ and $W$ have the same binomial exponents
if there are generating sets of $\mathbb{I}(V)$ and $\mathbb{I}(W)$ which have the property that
each generator of $\mathbb{I}(V)$ has the same binomial exponents 
as some generator of $\mathbb{I}(W)$, and 
vice versa. (For a precise formulation, see Definition~\ref{definition: same exponents}.)  Our main result is the following.

\begin{theorem}[Corollary \ref{productandpoint}]\label{maintheorem}
Let $V$ and $W$ be binomial varieties of $\mathbb{P}^n$ that
have the same binomial exponents.  Let $V' \subseteq V$ be any
subvariety.  If there exists $p = [p_0:\cdots:p_n] \in V'$ with
$p_0\cdots p_n \neq 0$, then 
$p \star W = V' \star W = V \star W.$
\end{theorem}

\noindent
The above result follows from the fact that the generators of
$\mathbb{I}(V \star W)$ can be computed directly from the generators of
$\mathbb{I}(V)$ and $\mathbb{I}(W)$.
The results of \cite{BC} and \cite{FOW}, mentioned above, can be recovered from Theorem \ref{maintheorem} (see Theorems \ref{thm:binomialhypersurface}, \ref{BCtype} and \ref{toricsquare}).

In a different direction, as 
another application of Theorem~\ref{maintheorem}, we prove the following result about 
the toric ideals of graphs.

\begin{theorem}[Theorem~\ref{toricidealresult}]
Let $G$ be a finite simple graph with
edge set 
$E = \{e_1,\ldots,e_q\}$ and suppose that $H$ is a subgraph of $G$ with edge set
$E' = \{e_{i_1},\ldots,e_{i_r}\}$.  
Let $I_G \subseteq R = k[e_1,\ldots,e_q]$ be the toric ideal of $G$, and let $I_H \subseteq
k[e_{i_1},\ldots,e_{i_r}]$ denote
the toric ideal of $H$.  If $I_H^e$ is the extension of $I_H$ in $R$, then $\mathbb{V}(I_G) \star \mathbb{V}(I^e_H) = \mathbb{V}(I^e_H)$.
\end{theorem}

We also compare some of the algebraic invariants of
$V$, $W$, and $V\star W$ when $V$ and $W$ are binomial varieties
with the same binomial exponents.  Under the condition
that $V$ and $W$ both contain at least one point with no zero
coordinates, we show that $\mathbb{I}(V), \mathbb{I}(W)$, and $\mathbb{I}(V \star W)$ all
have the same leading term ideal.  This implies that 
they all share the same Hilbert polynomial, and consequently,
$\deg(V) = \deg(W) = \deg(V \star W)$ and $\dim(V) = \dim(W)
=\dim(V\star W)$ (see Theorem \ref{mainresult3}).

Finally, we show that binomial varieties also  appear in a rather different context within the study of Hadamard products. More specifically, let $V$ be a fixed variety $V \subseteq \mathbb{P}^n$ (not necessarily a binomial variety)
and fix a point $p \in \mathbb{P}^n$  with no zero homogeneous coordinates.  Consider the  set 
\begin{equation}\label{eq: fiber} 
\psi(V,p) = \{q \in \mathbb{P}^n ~|~  q \star V = p \star V \}.
\end{equation} 
In other words, for a given Hadamard product $p \star V $, the set~\eqref{eq: fiber} consists of all points $q$ that yield the same result after taking the Hadamard product with $V$. 
The set $\psi(V,p)$ may not be an algebraic variety.
However, we show that under some mild hypothesis on 
the variety $V$, the ideal $\mathbb{I}(\psi(V,p))$ is generated
by binomials.  Furthermore, these binomials can 
be explicitly determined from a reduced Gr\"obner basis
of $\mathbb{I}(V)$ (see Theorem \ref{fiberofmap}).

Our paper is structured as follows.  In Section 2 we introduce
the relevant theory and background needed on Hadamard products.
In Section 3 we define binomial varieties with the same
exponents and prove our main result, Theorem~\ref{maintheorem} (Corollary \ref{productandpoint}).  In Section 4 we deduce 
some consequences of Theorem~\ref{maintheorem}, and we compute the dimension 
and degree of this family of varieties.  In the final section we show
how binomial varieties also appear when studying
the ideal of forms that vanish on the set $\psi(V,p)$
defined in \eqref{eq: fiber}.


\section{Background}\label{sec:background}

In this section, we recall the relevant definitions concerning Hadamard products 
and related results. Throughout this paper, $k$ denotes an algebraically closed field of  characteristic zero.  We let $\mathbb{P}^n$ denote the projective space over the field $k$ of dimension $n$, with homogeneous coordinates $[x_0:x_1:\cdots:x_n]$. Recall that a projective variety is a subset of $\mathbb{P}^n$ defined by the vanishing of a homogeneous ideal $I = \langle f_1, \ldots, f_m \rangle$ where each $f_i$ is a homogeneous polynomial in $R = k[x_0, x_1,\ldots, x_n]$ with respect to the standard grading.   If $X$ is a projective
variety, we will write $\mathbb{I}(X)$ for the defining (homogeneous)
ideal of $X$.  Similarly, given a homogeneous ideal $I$
of $R$, $\mathbb{V}(I)$ will denote the projective
variety described by $I$.  (By Hilbert's
Nullstellensatz, $\mathbb{V}(I) = \mathbb{V}(\sqrt{I})$,
where $\sqrt{I}$ is the radical of $I$, and 
$\mathbb{I}(X)$ is a radical ideal.)   If $I =
\langle f_1,\ldots,f_m\rangle$, we sometimes abuse
notation and write $\mathbb{V}(f_1,\ldots,f_m)$ instead
of $\mathbb{V}(\langle f_1,\ldots,f_m \rangle).$
Note that our 
projective varieties need not be irreducible, 
that is, $\mathbb{I}(X)$ does not need to be prime.

We summarize some of the needed results on Hadamard
products that are found in \cite{BCK,BC}. The Hadamard product is defined for two projectives varieties $X$ and $Y$, contained in the same projective space $\mathbb{P}^n$. Intuitively, the definition can be understood as being given by ``component-wise multiplication'' of the points in $X$ and $Y$. More precisely, we have the following. 

\begin{definition}[Hadamard product] Let $X,Y \subset \mathbb{P}^n$ be projective varieties. The \textit{Hadamard product of $X$ and $Y$}, denoted $X \star Y$, is the Zariski closure of the image of the rational map
\[X \times Y { \dashrightarrow} \ \mathbb{P}^n, \quad ([p_0: \cdots: p_n], [q_0: \cdots : q_n]) \mapsto [p_0q_0 : \cdots : p_n q_n]. \] 
Explicitly, we have 
\begin{equation}\label{eq: hadamard product} 
X \star Y := \overline{\{p \star q \mid p \in X,\ q \in Y,\ p \star q \textup{ is defined}\}} \subseteq \mathbb{P}^n
\end{equation} 
where $p \star q := [p_0q_0:\cdots:p_nq_n]$ is the point obtained by  component-wise multiplication of the points $p = [p_0 : \dots : p_n]$ and $q = [q_0 : \dots : q_n]$, and $p \star q$ is defined precisely when there exists at least one index $i$, $0 \leq i \leq n$, with $p_i q_i \neq 0$ (so that $p \star q = [p_0 q_0 : \cdots : p_n q_n]$ is a valid point in $\mathbb{P}^n$). 
\end{definition}

It is straightforward to use the definition of a Hadamard product to define the Hadamard power of a single projective variety. 

\begin{definition}[Hadamard power] Let $X \subset \mathbb{P}^n$ be a projective variety and let $r$ be a positive integer. The \textit{$r$-th Hadamard power} of $X$ is (inductively) defined as  \[X^{\star r} := X \star X^{\star (r-1)}\] where 
we use the convention that $X^{\star 0} :=[1:\cdots : 1]$ is a single point. 
\end{definition}

We can also define a Hadamard product
for homogeneous ideals, motivated by the definition of the Hadamard product of projective varieties.  

\begin{definition}\label{hadprodideals} (Hadamard product of ideals) 
Let $R$ be the polynomial ring $k[x_0, \ldots, x_n]$ in $n+1$ variables for a positive integer $n$, equipped with the standard grading. Let 
$I$ and $J$ be homogeneous ideals of $R$.   The \emph{Hadamard product of 
ideals} $I$ and $J$, denoted
$I \star J$, is the ideal constructed
via the following algorithm:
\begin{enumerate}
\item[$\bullet$] Define  
$S := k[x_0,\ldots,x_n,y_0,\ldots,y_n,z_0,\ldots,z_n]$.
\item [$\bullet$] Define  
$I({\bf y}) := \langle f(y_0,\ldots,y_n) ~|~
f(x_0,\ldots,x_n) \in I \rangle$ to be the 
ideal 
obtained by replacing $x_i$ with $y_i$
for all elements of $I$,
and similarly, let $J({\bf z})$
be the ideal obtained by 
replacing $x_i$ with $z_i$ for all elements
of $J$.
\item[$\bullet$] Define  
$K := I({\bf y}) + J({\bf z}) 
+ \langle x_0-y_0z_0,\ldots,x_n-y_nz_n \rangle
\subseteq S$.
\item[$\bullet$] Finally, define \begin{equation}\label{eq: def I J ideal Hadamard}
I \star J := K \cap k[x_0,\ldots,x_n].
\end{equation}
Intuitively, this last step ``eliminates the extra variables'' 
$y_0,\ldots,y_n, z_0, \ldots, z_n$ from the ideal $K$.
\end{enumerate}
\end{definition}

As first observed in \cite{BCK}, the Hadamard product of ideals corresponds precisely with the Hadamard product of projective varieties. 

\begin{lemma}[{\cite[Remark 2.6]{BCK}}]
\label{lemma: ideals and varieties equal} 
Let $X, Y \subseteq \mathbb{P}^n$ be
projective varieties with
defining (radical) ideals $\mathbb{I}(X)$ and $\mathbb{I}(Y)$ in $k[x_0,\ldots,x_n]$.  
Then $\mathbb{I}(X)\star \mathbb{I}(Y) = \mathbb{I}(X \star Y)$.  
\end{lemma} 

\noindent
In other words, Lemma~\ref{lemma: ideals and varieties equal} shows that the Hadamard product of the ideals
$\mathbb{I}(X)$ and $\mathbb{I}(Y)$ is precisely the defining
ideal of the Hadamard product $X \star Y$of the
varieties $X$ and $Y$.

Given $\alpha = (\alpha_0,\ldots,\alpha_n) \subseteq
\mathbb{N}^{n+1}$, we set $X^\alpha := 
x_0^{\alpha_0}x_1^{\alpha_1}\cdots x_n^{\alpha_n}$,
that is, $X^\alpha$ is a \emph{monomial} whose exponents are recorded in $\alpha$.   The \emph{degree} of $X^\alpha$ 
is $\deg(X^\alpha) := \alpha_0+\cdots+\alpha_n$.  
The following lemma, which will be required
in the next section, helps us describe some of
the elements contained in the
ideal $K$ used to define
$I\star J$ in Definition \ref{hadprodideals}.

\begin{lemma}\label{idealfact}
Let $S = k[x_0,\ldots,x_n,y_0,\ldots,y_n,z_0,\ldots,z_n]$
and consider the ideal 
$$\mathcal{I} := \langle x_0 -y_0z_0,\ldots,x_n-y_nz_n \rangle
\subseteq S.$$
Then for any $\alpha = (\alpha_0,\ldots,\alpha_n) \in \mathbb{N}^{n+1}$, we have 
$$X^\alpha - Y^\alpha Z^{\alpha} \in \mathcal{I}$$
where $X^\alpha = x_0^{\alpha_0} \cdots
x_n^{\alpha_n}$, $Y^\alpha = y_0^{\alpha_0} \cdots y_n^{\alpha_n}$, and
$Z^\alpha = z_0^{\alpha_0} \cdots 
z_n^{\alpha_n}$.
\end{lemma}

\begin{proof}
We proceed by induction on the number of nonzero entries
in $\alpha$.  For the base case, suppose that
$\alpha$ has only one nonzero entry, so 
$X^\alpha - Y^\alpha Z^{\alpha} = x_i^a - y_i^az_i^a$ for
some integer $i \in \{1,\ldots,n\}$ and $a \in \mathbb{N}$.  
If $a=1$, then the claim follows from the definition of $\mathcal{I}$.  If $a \geq 2$, then we can factor as follows 
$$x_i^a - y_i^az_i^a = (x_i-y_iz_i)(x_i^{a-1}+x_i^{a-2}y_iz_i + x_i^{a-3}y_i^2z_i^2 + \cdots + y_i^{a-1}z_i^{a-1}),$$
and since $x_i-y_iz_i$ is in $\mathcal{I}$ by definition, it follows that $x_i^a-y_i^az_i^a \in \mathcal{I}$ as desired. 

We proceed with the induction step. Suppose by induction that the claim holds if $\alpha$ has $t \geq 1$ many nonzero entries.  Now suppose that $\alpha$ has $t+1 > 1$ nonzero entries.  By a change of variables, we may assume without loss of generality that the nonzero entries occur in the first $t+1$ coordinates, i.e., 
$\alpha = (\alpha_0,\ldots,\alpha_t,0,\cdots,0)$.   Define 
$\alpha' := (\alpha_0,\ldots,\alpha_{t-1},0,\ldots,0)$.
Then we can factor as follows: 
\begin{eqnarray*}
X^\alpha - Y^\alpha Z^\alpha & = & x_t^{\alpha_t}(X^{\alpha'}- Y^{\alpha'} Z^{\alpha'}) + Y^{\alpha'} Z^{\alpha'}
(x_t^{\alpha_t} - y_t^{\alpha_t}z_t^{\alpha_t}).
\end{eqnarray*}
Since both $X^{\alpha'}
- Y^{\alpha'} Z^{\alpha'}$ and $x_t^{\alpha_t} - y_t^{\alpha_t}z_t^{\alpha_t}$ are in $\mathcal{I}$ by 
induction, we may conclude that $X^\alpha - Y^\alpha Z^\alpha$ is in $\mathcal{I}$ as desired. This completes the induction step, and hence the proof. 
\end{proof}

Given
a point $p =[p_0:\cdots:p_n] \in \mathbb{P}^n$
and $\alpha \subseteq \mathbb{N}^{n+1}$, we 
define $p^\alpha := p_0^{\alpha_0}\cdots p_n^{\alpha_n}$.

\begin{definition}\label{{dfn:hadamardtransf}}
(Hadamard transformation)
Let $f = \sum a_\alpha X^\alpha$ be a homogeneous polynomial
of $R$ and $p =[p_0:\cdots:p_n] \in \mathbb{P}^n$ be a 
point with no zero coordinates, i.e.,
$p_0\cdots p_n \neq 0$.  Then the \emph{Hadamard 
transformation} of $f$ by $p$, denoted by $f^{\star 
p}$, is the polynomial $f^{\star p} = \sum 
\frac{a_\alpha}{p^\alpha} X^\alpha$.
\end{definition}

The Hadamard transformation can be used to describe a generating set of the defining ideal of $p \star V$. 

\begin{theorem}[{\cite[Theorem 3.5]{BC}}]
\label{generatingset}
Let $V \subset \mathbb{P}^n$ be a projective variety and let 
$p = [p_0:\cdots:p_n]$ be a point with $p_0\cdots p_n \neq 0$.
If $\mathbb{I}(V) = \langle f_1,\ldots,f_s \rangle$, then
$$\mathbb{I}(p \star V) = 
\langle f_1^{\star p},\ldots, f_s^{\star p} \rangle.$$
Futhermore, if 
$\{f_1,\dots f_s\}$ is a Gr\"obner basis for 
$\mathbb{I}(V)$ with respect to a monomial order $<$, then  $\{f_1^{\star p},\dots,f_s^{\star p}\}$ 
is a Gr\"obner  basis for $\mathbb{I}(p \star V)$ with respect to the same
monomial order $<$.
\end{theorem}

Recall that a Gr\"obner bases $\mathcal{G} = \{f_1,\ldots,f_s\}$ of an ideal $I$ is a \emph{reduced} Gr\"obner basis if (1) the leading coefficient, that is, the coefficient of the leading term of  
each $f_i$ is $1$, and (2) no monomial appearing in 
$f_i$ is divisible by $LT(f_j)$, the lead term
of $f_j$, for any $j \neq i$.  
A key property of a reduced Gr\"obner bases 
of an ideal is that for a fixed monomial order, 
there is a unique reduced Gr\"obner basis (see
\cite[Chapter 2, Section 7, Theorem 5]{CLO}).   
We can now extend the previous theorem.

\begin{theorem} \label{reducedgb}
Let $V \subset \mathbb{P}^n$ be a projective variety and $p = [p_0:\cdots:p_n]$ with $p_0\cdots p_n \neq 0$.
Suppose 
$\{f_1,\dots f_s\}$ is a reduced Gr\"obner basis with respect to
 a monomial order $<$ for 
$\mathbb{I}(V)$ with $LT(f_i) = X^{\alpha_i}$ for $i=1,\ldots,s$. Then  $\{p^{\alpha_1}f_1^{\star p},\dots,p^{\alpha_s}f_s^{\star p}\}$ 
is a reduced 
Gr\"obner  basis for $\mathbb{I}(p \star V)$ with respect to the same
monomial order $<$.
\end{theorem}

\begin{proof}
By Theorem \ref{generatingset}, the set
$\{f_1^{\star p},\ldots,f_s^{\star p}\}$ is a 
Gr\"obner basis of $\mathbb{I}(p \star V)$.   
Since $f_i = X^{\alpha_i} + \mbox{(lower order terms)}$,
we have 
$f_i^{\star p} = \frac{1}{p^{\alpha_i}}X^{\alpha_i} + 
\mbox{(lower order terms)}.$
Thus $p^{\alpha_i}f_i^{\star p}$ has leading coefficient
$1$ for all $i$.  Moreover, since 
$LT(p^{\alpha_i}f_i^{\star p}) = p^{\alpha_i}LT(f_i^{\star p}) = 
X^{\alpha_i}$ for all $i=1,\ldots,s$, it is follows 
that $\{p^{\alpha_1}f_1^{\star p},\dots,p^{\alpha_s}f_s^{\star p}\}$ is 
also a Gr\"obner basis of $\mathbb{I}(p \star V)$.

So we need only to check that for each $1 \leq i \leq 
s$, the leading term of $p^{\alpha_i}f_i^{\star p}$ does not divide any term of $p^{\alpha_j}f_j^{\star p}$ for 
$i \neq j$. But since $\{f_1,\ldots,f_s\}$ is a reduced 
Gr\"obner basis, for each $i=1,\ldots,s$, 
the leading term of $f_i$ does not divide any term 
of $f_j$ for $i \neq j$. Furthermore, the terms of each 
$p^{\alpha_i}f_i^{\star p}$ are exactly same as 
the terms of $f_i$, simply with different coefficients. 
Thus, it must also be the case that the leading term of $p^{\alpha_i}f_i^{\star p}$ does not divide any term of $p^{\alpha_j}f_j^{\star p}$ for $i \neq j$.
\end{proof}

We will use Theorem~\ref{reducedgb} in our arguments below.


\section{Main results: binomial Ideals and Binomial Varieties}
\label{sec: binomial main}

As stated in the Introduction, the primary goal of this paper is to better understand
the Hadamard products of binomial ideals.  Our main
result is Theorem \ref{mainresult1} below.  This result then allows us
to recover some of the known results in the literature
as special cases (see Section 4).

We begin by recalling
the relevant definitions and notation.
We take the ambient ring to be $R =k[x_0, \ldots, x_n]$ where, as before, $k$ is an algebraically closed field of characteristic $0$. 

\begin{definition}(Binomials, binomial ideal,
binomial variety)  
A \emph{binomial} is a polynomial with exactly
two terms, that is, $a_1X^{\alpha_1} - a_2 X^{\alpha_2}$ where
$0 \neq a_i \in k$ for $i=1,2$, and $\alpha_1, \alpha_2 \in \mathbb{N}^{n+1}$ with $\alpha_1 \neq \alpha_2$.   We say that a binomial
is \emph{homogeneous} if $\deg(X^{\alpha_1}) = \deg(X^{\alpha_2})$.
An  ideal $I \subseteq R$ is a \emph{(homogeneous) 
binomial ideal} if $I$ is generated by
(homogeneous) binomials.
A projective variety $V \subseteq \mathbb{P}^n$ 
is a \emph{binomial variety} if $\mathbb{I}(V)$ is a homogeneous binomial ideal.  
\end{definition}

\begin{remark}\label{binomialremark}
In some references, notably \cite{HHO}, a binomial is defined to be the difference of two
monomials, that is, $a_1 = a_2 =1$. In other sources, for example \cite{ES}, one of $a_1$ and $a_2$ is 
allowed to be zero, and consequently, a monomial is also considered to be a binomial.  Our definition for a binomial only allows 
for nonzero coefficients, and we allow the coefficients to be different from $1$.
\end{remark}

Given two binomial varieties $V$ and $W$
in $\mathbb{P}^n$, the next definition
provides terminology to compare the binomials
in the generating sets of $\mathbb{I}(V)$ and $\mathbb{I}(W)$.

\begin{definition}(Same binomial exponents)\label{definition: same exponents}
Two binomial varieties $V$ and $W$ 
\emph{have the same binomial exponents} 
if there are two ordered subsets 
$\{\alpha_1,\ldots,\alpha_s\}$ and $\{\beta_1,\ldots,\beta_s\}$ of ${\mathbb{N}}^{n+1}$  such that $\alpha_i \neq \beta_i$ for all $i=1,\ldots,s$, the pairs $(\alpha_i, \beta_i)$ of exponents are pairwise distinct for all $i=1,\ldots,s$, and there are nonzero constants 
$a_1,\ldots,a_s,b_1,\ldots,b_s,c_1,\ldots,c_s,
d_1,\ldots,d_s\in k \setminus \{0\}$
such that
$$\mathbb{I}(V) = \langle a_1X^{\alpha_1} -b_1X^{\beta_1},
a_2X^{\alpha_2} - b_2X^{\beta_2},\ldots, a_sX^{\alpha_s}-b_sX^{\beta_s} \rangle$$
and 
$$\mathbb{I}(W) =
\langle c_1X^{\alpha_1} -d_1X^{\beta_1},
c_2X^{\alpha_2} - d_2X^{\beta_2},\ldots, c_sX^{\alpha_s}-d_sX^{\beta_s} \rangle.
$$
\end{definition}

In other words, two binomial varieties have
the same binomial exponents if the only
difference between the sets of binomials that generate
the corresponding ideals are in the coefficients appearing in the binomials.  Note that, by definition, a
binomial variety has the same binomial exponents
with itself.

\begin{lemma}\label{firstcontainment}
Let $V$ and $W$ be binomial varieties
of $\mathbb{P}^n$. Assume that $V$ and $W$ have the same binomial
exponents with ideals $\mathbb{I}(V)$ and $\mathbb{I}(W)$ as in Definition~\ref{definition: same exponents}. 

Then the ideal 
$$ 
 \langle a_1c_1X^{\alpha_1} -b_1d_1X^{\beta_1},
a_2c_2X^{\alpha_2} - b_2d_2X^{\beta_2},\ldots, a_sc_sX^{\alpha_s}-b_sd_sX^{\beta_s} \rangle$$
is contained in the ideal 
$\mathbb{I}(V\star W)$. 
\end{lemma}

\begin{proof} 
In the ring $S= k[x_0,\ldots,x_n,y_0,\ldots,y_n,
z_0,\ldots,z_n]$, we have
the identity
\begin{eqnarray*}
a_ic_iX^{\alpha_i} - b_id_iX^{\beta_i} & = & 
a_ic_i(X^{\alpha_i}-Y^{\alpha_i}Z^{\alpha_i})
- b_id_i(X^{\beta_i} - Y^{\beta_i}Z^{\beta_i}) +\\
&& c_iZ^{\alpha_i}(a_iY^{\alpha_i} - b_iY^{\beta_i})
+ b_iY^{\beta_i}(c_iZ^{\alpha_i} - d_iZ^{\beta_i}) 
\end{eqnarray*}
for each $i = 1,\ldots,s$.  Note that it follows from Lemma~\ref{idealfact} that the first two terms on
the right hand side of the above expression 
are elements of 
the ideal $K$ in Definition \ref{hadprodideals}.
Moreover, by definition of the ideals $\mathbb{I}(V)({\bf y})$ (respectively $\mathbb{I}(W)({\bf z})$), 
$a_iY^{\alpha_i} - b_iY^{\beta_i}$ (respectively $c_iZ^{\alpha_i} - d_iZ^{\beta_i}$)
is in $\mathbb{I}(V)({\bf y})$ (respectively $\mathbb{I}(W)({\bf z})$). It follows that $a_ic_iX^{\alpha_i} - b_i d_i X^{\beta_i}$ is in $K$. 

Since $a_ic_iX^{\alpha_i} - b_id_iX^{\beta_i}$ 
contains only the variables $x_0,\ldots,x_n$, it follows from  
Definition \ref{hadprodideals} that
$$a_ic_iX^{\alpha_i} - b_id_iX^{\beta_i} 
\in K \cap k[x_0,\ldots,x_n] =: \mathbb{I}(V) \star \mathbb{I}(W) \mbox{ for all $i=1,\ldots,s$}.$$ 
Because $\mathbb{I}(V) \star \mathbb{I}(W) = \mathbb{I}(V \star W)$ by
Lemma \ref{lemma: ideals and varieties equal}, 
this completes the proof.
\end{proof}

We can now prove our first main result.

\begin{theorem}\label{mainresult1}
Let $V$ and $W$ be binomial varieties
of $\mathbb{P}^n$. Assume that $V$ and $W$ have the same binomial
exponents.  In addition,
suppose that $V$ or $W$ contains a point 
$p = [p_0:\cdots:p_n]$
with $p_0\cdots p_n \neq 0$.
Then $V \star W$ is also a binomial
variety that has the same binomial exponents
as $V$ and $W$.   More precisely,
if 
$$\mathbb{I}(V) = \langle a_1X^{\alpha_1} -b_1X^{\beta_1},
a_2X^{\alpha_2} - b_2X^{\beta_2},\ldots, a_sX^{\alpha_s}-b_sX^{\beta_s} \rangle$$
and 
$$\mathbb{I}(W) =
\langle c_1X^{\alpha_1} -d_1X^{\beta_1},
c_2X^{\alpha_2} - d_2X^{\beta_2},\ldots, c_sX^{\alpha_s}-d_sX^{\beta_s} \rangle,
$$
then
\begin{equation}\label{samebinomial}
\mathbb{I}(V\star W) = 
 \langle a_1c_1X^{\alpha_1} -b_1d_1X^{\beta_1},
a_2c_2X^{\alpha_2} - b_2d_2X^{\beta_2},\ldots, a_sc_sX^{\alpha_s}-b_sd_sX^{\beta_s} \rangle.
\end{equation}
\end{theorem}

\begin{proof} 
It suffices to verify the equality 
\eqref{samebinomial}, since the first conclusion follows immediately from  this 
statement. By Lemma \ref{firstcontainment}, it 
is enough to show that the containment
$\subseteq$ holds.

Without loss of generality, suppose
that $p = [p_0:\cdots:p_n]$ with
$p_0 \cdots p_n \neq 0$ is in $V$.
By Theorem \ref{generatingset}, 
\begin{eqnarray*}
\mathbb{I}(p \star W) &=&
\langle (c_1X^{\alpha_1} -d_1X^{\beta_1})^{\star p},
(c_2X^{\alpha_2} - d_2X^{\beta_2})^{\star p},\ldots, 
(c_sX^{\alpha_s}-d_sX^{\beta_s})^{\star p} \rangle \\
&&
\left\langle
\frac{c_1}{p^{\alpha_1}}X^{\alpha_1} -
\frac{d_1}{p^{\beta_1}}X^{\beta_1},
\ldots, 
\frac{c_s}{p^{\alpha_s}}X^{\alpha_s} -
\frac{d_s}{p^{\beta_s}}X^{\beta_s}
\right\rangle.
\end{eqnarray*}
Since $p \in V$, we also know that 
$$
a_ip^{\alpha_i} - b_ip^{\beta_i} =0 ~~
\mbox{so that}~~ p^{\alpha_i} = \frac{b_i}{a_i}p^{\beta_i} ~~
\mbox{for all $i=1,\ldots,s$}$$
where we are also using that $a_i \neq 0$ for all $i=1,\ldots,s$. 
Using these identities, we have
\begin{eqnarray*}
\mathbb{I}(p \star W) &=&
\left\langle
\frac{c_1}{\frac{b_1}{a_1}p^{\beta_1}}X^{\alpha_1} -
\frac{d_1}{p^{\beta_1}}X^{\beta_1},
\ldots, 
\frac{c_s}{
\frac{b_s}{a_s}p^{\beta_s}
}X^{\alpha_s} -
\frac{d_s}{p^{\beta_s}}X^{\beta_s}
\right\rangle.
\end{eqnarray*}
Clearing denominators then gives
$$\mathbb{I}(p\star W) =
\langle a_1c_1X^{\alpha_1} -b_1d_1X^{\beta_1},
a_2c_2X^{\alpha_2} - b_2d_2X^{\beta_2},\ldots, a_sc_sX^{\alpha_s}-b_sd_sX^{\beta_s} \rangle.
$$
This shows that $\mathbb{I}(p \star W)$ is equal to the RHS of~\eqref{samebinomial}. 
Since $p \star W \subseteq V \star W$,
we have $\mathbb{I}(p \star W) \supseteq \mathbb{I}(V\star W)$, so $\mathbb{I}(V \star W)$ is contained in the RHS of~\eqref{samebinomial}, 
as desired. 
\end{proof}

\begin{example}
We illustrate some of the previous definitions and results 
with an example in $\mathbb{P}^2$.  Let 
$R = k[x,y,z]$ be the associated coordinate ring,
and suppose that 
$$I = \langle x^3 - 2y^2z \rangle ~~\mbox{and}~~
J = \langle x^3 - 2yz^2 \rangle.$$
Note that  the exponents that appear in the 
binomial generators of $I$ and $J$ are not the same. 

Using Definition \ref{hadprodideals}, we can compute
$I \star J$ using Macaulay2 \cite{M2} to find
that $I \star J = \langle 0 \rangle$.  If $V = \mathbb{V}(I)$
and $W = \mathbb{V}(J)$, we thus have
$$V\star W = \mathbb{V}(I) \star \mathbb{V}(J) 
= \mathbb{V}(I \star J) = \mathbb{P}^2.$$
Note that $I$ and $J$ do not have the same exponents, so Theorem~\ref{mainresult1} does not
apply.

On the
other hand, certainly the ideal $I$ has the same binomial exponents as itself. If we compute $I \star I$ using Macaulay 2, we get 
$I \star I = \langle x^3 - 4y^2z \rangle$. This agrees with the conclusion of Theorem~\ref{mainresult1}. 
\end{example}

Next, it is worth noting that our proof of Theorem~\ref{mainresult1} actually gives a stronger result. More specifically, it shows that if a binomial variety $V$ contains a point $p$ with no zero coordinates, then for any other binomial variety $W$
with the same binomial exponents as $V$, we have that the Hadamard products of $W$ with $V$ and $p$ are equal, i.e., $V \star W = p \star W$. 
We can refine this observation even further to obtain the following corollary.

\begin{corollary}\label{productandpoint}
Let $V$ and $W$ be binomial varieties
of $\mathbb{P}^n$. Assume that $V$ and $W$ have the same binomial
exponents.   Let $V' \subseteq V$ be any
subvariety.  If $V'$ contains
a point
$p = [p_0:\cdots:p_n]$
with $p_0\cdots p_n \neq 0$, then  
$$p \star W = V' \star W = V \star W.$$
\end{corollary}

\begin{proof}
Since $p \in V' \subseteq V$, it clearly
follows that $p\star W \subseteq V' \star W
\subseteq V \star W$.  Consequently,
$\mathbb{I}(p\star W) \supseteq  \mathbb{I}(V' \star W)
\supseteq \mathbb{I}(V \star W)$.  On the other hand,
the proof of Theorem \ref{mainresult1}
showed that $\mathbb{I}(p\star W) = \mathbb{I}(V \star W)$.
The desired conclusion now follows.
\end{proof}

Using Theorem \ref{generatingset} and some additional hypotheses, we can also give 
a Gr\"obner basis for $\mathbb{I}(V \star W)$. 

\begin{corollary}
Let $V$ and $W$ be binomial varieties
of $\mathbb{P}^n$. Assume that $V$ and $W$ have the same binomial
exponents.  Suppose that $V$ contains
a point
$p = [p_0:\cdots:p_n]$
with $p_0\cdots p_n \neq 0$.  If  
$$\mathbb{I}(V) = \langle a_1X^{\alpha_1} -b_1X^{\beta_1},
a_2X^{\alpha_2} - b_2X^{\beta_2},\ldots, a_sX^{\alpha_s}-b_sX^{\beta_s} \rangle$$
and if
$$\{c_1X^{\alpha_1} -d_1X^{\beta_1},
c_2X^{\alpha_2} - d_2X^{\beta_2},\ldots, c_sX^{\alpha_s}-d_sX^{\beta_s} \}$$
is a Gr\"obner basis of $\mathbb{I}(W)$
with respect to some monomial
order $<$, then
$$\{ a_1c_1X^{\alpha_1} -b_1d_1X^{\beta_1},
a_2c_2X^{\alpha_2} - b_2d_2X^{\beta_2},\ldots, a_sc_sX^{\alpha_s}-b_sd_sX^{\beta_s} \}$$
is a Gr\"obner basis of $\mathbb{I}(V\star W)$ with respect to the monomial order
$<$.
\end{corollary}

\begin{proof}
First note that for any Gr\"obner basis $\{f_1,\ldots,f_s\}$ of an ideal $I$,
the set $\{e_1f_1,\ldots,e_sf_s\}$ is also a Gr\"obner basis for $I$ with the
same monomial order if $e_1,\ldots,e_s$ are nonzero elements of $k$.

By Theorem \ref{generatingset}, the set 
$$\{(c_1X^{\alpha_1} -d_1X^{\beta_1})^{\star p},
(c_2X^{\alpha_2} - d_2X^{\beta_2})^{\star p},\ldots,
(c_sX^{\alpha_s}-d_sX^{\beta_s})^{\star p} \} $$
is a Gr\"obner basis of $\mathbb{I}(p \star W)$.  
Since $\mathbb{I}(p\star W) =\mathbb{I}(V\star W)$ by
Corollary \ref{productandpoint}, the result
now follows from Theorem \ref{mainresult1}, the proof of which shows that
$a_ic_iX^{\alpha_i} - b_id_iX^{\beta_i}$ is a nonzero scalar multiple of
$ (c_iX^{\alpha_i} - d_iX^{\beta_i})^{\star p}$
for $i =1,\ldots,s$.
\end{proof}


\section{Applications of main results}

In this section we record several consequences
of the results of the previous section. Before proceeding, we first prove 
a preliminary lemma related to the hypothesis (repeatedly invoked in the results of the previous section) that there exists 
a point $p = [p_0:\cdots:p_n]$ in $V$ 
with $p_0\cdots p_n \neq 0$.

\begin{lemma}\label{no_monomial}
Let $V$ be a binomial variety in $\mathbb{P}^n$.  If
there is a point $p = [p_0:\cdots:p_n] \in V$ such
that $p_0\cdots p_n \neq 0$, then $\mathbb{I}(V)$ contains no monomial.
\end{lemma}

\begin{proof}
If there exists a monomial $X^\alpha \in \mathbb{I}(V)$, 
then $p^\alpha =0$.  But this cannot happen since 
no coordinate of $p$ is zero.
\end{proof}

\subsection{Binomial hypersurfaces and
potent binomial varieties}
The results of Section 3 allow us to recover some
results of Bocci and Carlini \cite{BC}.

In the statement below, a projective variety $V \subseteq 
\mathbb{P}^n$ is a \emph{binomial
hypersurface} if $\mathbb{I}(V)$ is generated by
a single irreducible binomial.  The following result
was first proved by Bocci and Carlini. We can now obtain it as a consequence of our Theorem \ref{mainresult1}.

\begin{theorem}
[{\cite[Proposition 4.5]{BC}}]
\label{thm:binomialhypersurface}
Let $V$ and $W$ be binomial
hypersurfaces of $\mathbb{P}^n$, with
$\mathbb{I}(V)  = \langle a X^{\alpha} - bX^{\beta})$
and $\mathbb{I}(W) = 
\langle c X^{\alpha} - d X^{\beta})$, where $a,b,c,d$ are nonzero. 
Then the Hadamard product $V \star W$ is also a binomial hypersurface 
with defining ideal
$$\mathbb{I}(V \star W) = \langle
ac X^{\alpha} - bd X^{\beta}
\rangle.$$
\end{theorem}

\begin{proof}
We first observe that since $aX^{\alpha} - bX^{\beta}$ is irreducible by assumption,
a variable $x_i$ can divide at most one of $X^{\alpha}$ or $X^{\beta}$.  Indeed,
if $x_i$ divided both $X^{\alpha}$ and $X^{\beta}$, then we can
factor out the $x_i$, contradicting the irreducibility of the binomial.
So, we can write the binomial $aX^\alpha - bX^\beta$ as
$$ax^{\alpha_{i_1}}_{i_1}\cdots x^{\alpha_{i_s}}_{i_s} - bx^{\beta_{j_1}}_{j_1}\cdots x^{\beta_{j_s}}_{j_s}  ~~\mbox{with
$\alpha_{i_k},\beta_{j_l }\geq 1$}.$$
and with $\{i_1,\ldots,i_s\} \cap \{j_1,\ldots,j_s\} = \emptyset$.   Without
loss of generality, we can assume that $i_1 = 0$, i.e., the binomial
is 
$$ax^{\alpha_0}_0x^{\alpha_{i_2}}_{i_2}\cdots x^{\alpha_{i_s}}_{i_s} - bx^{\beta_{j_1}}_{j_1}\cdots x^{\beta_{j_s}}_{j_s}.$$
Then the point 
$$p = \left[ \sqrt[\alpha_0]{\frac{b}{a}}:1:1:\cdots:1\right]$$
vanishes at this binomial. (Here we use  $\sqrt[\alpha_0]{\frac{b}{a}}$ to denote any choice of $\alpha_0^{\rm{th}}$ root of $\frac{b}{a}$.) Moreover, because this is a point
on $V$ with no zero coordinates, the fact that
$\mathbb{I}(V \star W) = \langle acX^\alpha - bdX^\beta \rangle $ then
follows from Theorem \ref{mainresult1}.

Finally, we verify that $acX^\alpha-bdX^\beta$ is irreducible.  In order to simplify our notation, we assume 
without loss of generality that 
$i_1 = 0$ and $j_1 = 1$, and thus
$$acX^\alpha-bdX^\beta = acx^{\alpha_0}_0x^{\alpha_{i_2}}_{i_2}\cdots x^{\alpha_{i_s}}_{i_s} - bdx^{\beta_1}_{1}x^{\beta_{j_2}}_{j_2}\cdots x^{\beta_{j_s}}_{j_s}.$$
If $ac X^\alpha  - bd X^\beta = f(x_0,\ldots,x_n)g(x_0,\ldots,x_n)$ was
reducible, then note that
$$aX^\alpha - bX^\beta = 
f\left(\frac{x_0}{\sqrt[\alpha_0]{c}},
\frac{x_1}{\sqrt[\beta_1]{d}},x_2,\ldots,
x_n\right)
g\left(\frac{x_0}{\sqrt[\alpha_0]{c}},
\frac{x_1}{\sqrt[\beta_1]{d}},x_2,\ldots,
x_n\right)$$
is also reducible, contradicting our assumptions.
\end{proof}

\begin{remark}
\cite[Proposition 4.7]{BC} shows that the converse holds; 
that is, if $V$ and $W$
are hypersurfaces such that $V\star W$ is also
a hypersurface, then $V$ and $W$ are binomial hypersurfaces
with the same binonomial exponents.    
\end{remark}

For $i = 1,\ldots,s$
let $\xi_i$ be a primitive $(t_i-1)$-th
root of unity, and let $1 \leq \epsilon_i
\leq t_i-1$.  Following \cite[Definition 6.1]{BC}, a binomial variety
$V \subseteq \mathbb{P}^n$ is
said to have \emph{type $[(t_1,\epsilon_1),
\ldots,(t_s,\epsilon_s)]$} if
$$\mathbb{I}(V) = 
\langle X^{\alpha_1} - \xi_1^{\epsilon_1}X^{\beta_1},
\ldots, X^{\alpha_s} - \xi_s^{\epsilon_s}X^{\beta_s}
\rangle.
$$
Binomial varieties of this form are 
potent under the Hadamard product, i.e., there exists
an integer $t$ such that $V^{\star t} = V$.
This result, first shown by
Bocci and Carlini, can also be obtained
by our Theorem \ref{mainresult1}.

\begin{theorem}[{\cite[Theorem 6.2]{BC}}]
\label{BCtype}
Let $V \subseteq \mathbb{P}^n$ be a binomial
variety of type $[(t_1,\epsilon_1),
\ldots,(t_s,\epsilon_s)]$.  Then
$V^{\star t} = V$ where 
$t -1 = {\rm lcm}\left(\frac{t_1}{{\rm gcd}(t_1,\epsilon_1)},\ldots,\frac{t_s}{{\rm gcd}(t_s,\epsilon_s)}\right)$.
\end{theorem}

\begin{proof}
Note that for each $(t_i-1)$-th 
primitive root $\xi_i$,  we have $(\xi_i^{\epsilon_i})^t =
\xi_i^{\epsilon_i}$.  Thus, by
Theorem \ref{mainresult1} we have
$$
\mathbb{I}(V^{\star t}) =
\underbrace{\mathbb{I}(V) \star \cdots 
\star \mathbb{I}(V)}_t =
\langle X^{\alpha_1} - (\xi_1^{\epsilon_1})^t X^{\beta_1},
\ldots, X^{\alpha_s} - (\xi_s^{\epsilon_s})^t X^{\beta_s}\rangle =
\mathbb{I}(V).
$$
The conclusion $V^{\star t} =V$ now
follows.
\end{proof}
\subsection{Toric varieties and toric ideals of graphs}

We now turn our attention to a special type of binomial ideal. By \cite[Theorem 3.4]{HHO}
an ideal $I$ is a \emph{toric ideal}
if $I$ is a prime ideal and generated by 
binomials of the form $X^\alpha - X^\beta$, i.e.,
one monomial term has coefficient $1$ and the other monomial term has coefficient $-1$. Following Sturmfels \cite[Chapter 4]{S},
if a toric ideal $I$ is also homogeneous, then the
variety $\mathbb{V}(I)$ is a \emph{projective toric
variety}. As Sturmfels notes (see pg. 31 of \cite{S}), some definitions of projective
toric varieties also require the variety to be normal.  
In this paper, we do not require $\mathbb{V}(I)$ to be normal, only that
$I$ is a prime ideal generated by binomials of a certain form.

Our main results in Section~\ref{sec: binomial main} allow us to 
recover a result of  Friedenberg, Oneto, and Williams
\cite[Proposition 4.7]{FOW} on toric varieties; in fact, our theorem below is slightly more general than the statement in \cite{FOW}, which deals only with the special case $V'=V$.

\begin{theorem}
\label{toricsquare}
Suppose that $V \subseteq \mathbb{P}^n$ is a toric variety
and $V' \subseteq V$ is a subvariety
with a point $p = [p_0:\cdots:p_n] \in V'$
with $p_0\cdots p_n \neq 0$.
Then
$V' \star V = V$. In particular, in the case $V'=V$, then $V^{\star 2} = V$.
\end{theorem}

\begin{proof}
Since $V$ is a toric variety, the defining ideal $\mathbb{I}(V)$ of $V$
has the form
$$\mathbb{I}(V) = \langle X^{\alpha_1}-X^{\beta_1},\ldots,
X^{\alpha_s} - X^{\beta_s} \rangle.$$
 Applying
Corollary \ref{productandpoint} in the case when $W$ is equal to $V$, we can conclude 
$V' \star V = V \star V$. 

To see that $V \star V = V$, note that 
$[1:\cdots:1] \in V$ since this point
vanishes at each generator of $\mathbb{I}(V)$. By Theorem \ref{mainresult1},  since $a_i = b_i =1$ for all
$i=1,\ldots,s$, we have $\mathbb{I}(V \star V) = \mathbb{I}(V)$, 
and consequently, $V \star V = V$.
\end{proof}

We next study \emph{toric ideals of graphs} (to be defined precisely in Definition~\ref{def: toric ideal of graph} below) and their associated varieties. These ideals are defined using the combinatorial data of a finite simple graph and are an actively studied topic in combinatorial commutative algebra (see, for example, \cite{HHO,V}).
It turns out that the varieties defined by toric ideals of graphs are examples of toric
varieties, i.e., toric ideals of graphs are prime binomial ideals whose generators are differences of monomials. This allows us to use the results we obtained above to prove properties about the Hadamard products of toric ideals of graphs (and their subgraphs). 

We briefly recall the relevant background. We begin with the definition of the combinatorial objects in question. 

\begin{definition}(Finite simple graph, subgraph)  
A \emph{finite simple graph} is an ordered pair $G=(V,E)$, where $V$ is the set of vertices, $V$ is finite, and $E\subseteq \{\{u,v\}\mid u,v \in V, u\neq v\}$ is the set of edges. Note that, by the definition of $E$, we implicitly assume that there are no edges of the form $\{v,v\}$ for a vertex $v \in V$, and that there cannot exist more than one edge between two fixed distinct vertices $u \neq v$. 
We say that a finite simple graph $H=(V_1,E_1)$ is a \emph{subgraph} of $G$ if $V_1 \subseteq V$ and $E_1 \subseteq E$. 
\end{definition}

To describe the generators of the toric ideal of a graph, we need
the notion of a walk.

\begin{definition}(Walk, closed walk, even walk)
Let $G=(V,E)$ be a finite simple graph. A \emph{walk} of $G$ is a sequence of edges $w = (e_1, e_2, \dots, e_m)$, where each $e_i=\{u_{i_1},u_{i_2}\}\in E$ and $u_{i_2}=u_{{(i+1)}_1} \in V$ for each $i = 1, \dots, m - 1$. 
Equivalently, a walk is a sequence of vertices $(u_1, \dots,u_m,u_{m+1})$ such that $\{u_i,u_{i+1}\} \in E$ for all $i = 1,\dots,m$. A walk is \emph{even} if $m$ is even. A walk is \emph{closed} if $u_{m+1}=u_1$. 
\end{definition}

We are now ready to define the relevant ideals.  

\begin{definition}\label{def: toric ideal of graph}(Toric ideal of a graph)
Let $G=(V,E)$ be a finite simple graph. Suppose $\lvert V \rvert=n$ and denote $V=\{v_1,\dots,v_n\}$ and suppose $\lvert E \rvert=q$ and denote $E=\{e_1,\dots,e_q\}$.  Consider the following ring homomorphism, defined as \[
\varphi \colon k[e_1,\dots,e_q]\to k[v_1,\dots,v_n], \quad e_i \mapsto \varphi(e_i) := v_{i_1} v_{i_2} \quad \textup{ for all } e_i = \{v_{i_1}, v_{i_2}\}, 1 \leq i \leq q.
\] 
The \emph{toric ideal of $G$}, denoted $I_G$, is defined to be the kernel of the homomorphism $\varphi$.
\end{definition}

We defined a toric ideal at the beginning of this section as a prime ideal that is generated by certain binomials. In Definition~\ref{def: toric ideal of graph}, we define a certain ideal $I_G$ associated to $G$, and we call it a toric ideal. In order to show that this nomenclature is accurate, we need now to show that $I_G$ is prime and that it is generated by binomials of the form $X^\alpha - X^\beta$ with $\alpha \neq \beta$, i.e., the binomials have coefficient $1$. The next result, quoted from \cite{HHO}, is a key part of justifying this claim. More specifically, the result below shows that a generating set of $I_G$ can be described in terms
of the closed even walks of $G$, and furthermore,
these generators are binomials of the required form. We have the following.

\begin{theorem}[{\cite[Lemma 5.9]{HHO}}]
\label{toricgenerators}
Let $\Gamma=(e_{i_1},\dots,e_{i_{2m}})$ be a closed even walk of a finite simple graph $G$. Define the binomial \[f_\Gamma =\prod_{2 \nmid j} e_{i_j} -\prod_{2 \mid j} e_{i_j}.\]
Then the toric ideal $I_G$ is generated by all the binomials $f_\Gamma$, where $\Gamma$ is a closed even walk of $G$.
\end{theorem}

In order to complete the reasoning that $I_G$ is a toric ideal, we should justify that $I_G$ is in fact prime. 
Indeed, $I_G$ is prime since it is the preimage of the zero ideal, which is a prime ideal in the integral domain 
$k[v_1,\ldots,v_n]$.

We can now derive a result on toric ideals of graphs (Theorem~\ref{toricidealresult}), using our Theorem~\ref{toricsquare} on general toric ideals.
Some preparatory remarks are necessary. Given $G$ a finite simple graph and $H$ a subgraph, let the set of edges of $G$ be $E = \{e_1,\ldots, e_q\}$ and let the subset of $E$ giving the edges of the subgraph $H$ be denoted $E' = \{e_{i_1}, \ldots, e_{i_r}\}$ for $e_{i_\ell}\in E$ for all $1 \leq \ell \leq r$. Definition~\ref{def: toric ideal of graph} constructs two ideals $I_G \subseteq k[e_1, \ldots, e_q]$ and $I_H \subseteq k[e_{i_1}, \ldots, e_{i_r}]$. It is clear that there is a natural inclusion $\Psi$ from the ambient ring $k[e_{i_1},\ldots,e_{i_r}]$ of $I_H$ into $k[e_1,\ldots, e_q]$, and hence we may consider the natural extension $I_H^e$ of $I_H$ to $k[e_1,\ldots,e_q]$, defined by $I_H^e := \langle \Psi(I_H) \rangle$, i.e.,  the ideal generated by the image of $I_H$ under $\Psi$. We have the following.

\begin{theorem}\label{toricidealresult}
Let $G$ be a finite simple graph with
edge set 
$E = \{e_1,\ldots,e_q\}$ and suppose that $H$ is a subgraph of $G$ with edge set
$E' = \{e_{i_1},\ldots,e_{i_r}\}$.  
Let $I_G \subseteq k[e_1,\ldots,e_q]$ be the toric ideal of $G$, and let $I_H \subseteq
k[e_{i_1},\ldots,e_{i_r}]$ denote
the toric ideal of $H$.  If $I_H^e$ is the extension of $I_H$ defined above, then $I_G \star I^e_H = I^e_H$.
\end{theorem}

\begin{proof} 
We saw in the discussion above that $I_H$, and hence also $I_H^e$, is a toric ideal. Hence $\mathbb{V}(I_H^e)$ is a toric variety by definition.
Moreover, since 
every closed even walk in $H$ is also a closed even
walk of $G$, Theorem \ref{toricgenerators} implies
that $I^e_H \subseteq I_G$.   This in turn implies that $\mathbb{V}(I_G)$ is a subvariety of 
the toric variety $\mathbb{V}(I^e_H)$. 
From the form of the generators of $I_G$ given in Definition~\ref{def: toric ideal of graph}, the point $[1:\cdots:1]$ lies in the vanishing locus of every generator of $I_G$, so $[1:\cdots:1] \in \mathbb{V}(I_G)$.  By Theorem 
\ref{toricsquare}, we may conclude that $\mathbb{V}(I_G) \star \mathbb{V}(I^e_H) = \mathbb{V}(I^e_H)$.
Since toric ideals are prime and hence radical, we know $I_G = \mathbb{I}(\mathbb{V}(I_G))$ and similarly $I_H^e = \mathbb{I}(\mathbb{V}(I_H^e))$. Hence Lemma~\ref{lemma: ideals and varieties equal} implies $I_G \star I^e_H  = I^e_H$, as desired. 

\end{proof}

\subsection{Hilbert functions} 

Suppose that $V$ and $W$ are two binomial varieties in $\mathbb{P}^n$
with the same binomial exponents. Suppose in addition that both $V$ and $W$
contain at least one point with no zero coordinates.  The main result of this section is Theorem~\ref{mainresult3}, which shows that, in the setting just described, the three varieties $V$, $W$, and $V \star W$ 
all have the same Hilbert function, and consequently, also have the same
degree and dimension.

To prove Theorem~\ref{mainresult3}, we use the theory of monomial orders and Gr\"obner bases.  (We refer the reader to \cite{CLO} for terminology and background concerning Gr\"obner bases, monomial
orders, and initial ideals.) Specifically, we show in Proposition~\ref{leadingterm} below that, with respect to any monomial order $<$, the ideals $\mathbb{I}(V), \mathbb{I}(W)$ and 
$\mathbb{I}(V \star W)$ all have the same initial ideal. This is a key step in the argument for Theorem~\ref{mainresult3}. To place Proposition~\ref{leadingterm} in context, recall first that it is well-known (see e.g. \cite[Proposition 1.1]{ES}) that a Gr\"obner basis of a binomial 
ideal consists of binomials.
Our result below
refines this well-known statement by showing that, if $I$ and $J$ are binomial and have the same exponents, then their respective Gr\"obner bases (with respect to any monomial order) are not only binomial, but also have the same exponents. We have the following.  

\begin{proposition}\label{leadingterm}
Let $V$ and $W$ be two binomial varieties in $\mathbb{P}^n$
that have the same binomial exponents.  Suppose that both
$V$ and $W$ contain at least one point with no zero coordinates. Let $<$ be any monomial order on $R=k[x_0,x_1,\ldots,x_n]$. 
Then the initial ideals $LT_<(I)$ and $LT_<(J)$
of $I =\mathbb{I}(V) $ and $J = \mathbb{I}(W)$, respectively, are equal. That is, 
$$LT_<(I) = \langle LT_<(f) ~|~ f \in I \rangle = 
\langle LT_<(g) ~|~ g \in J \rangle = LT_<(J).$$
\end{proposition}

\begin{proof}
This is straightforward from Buchberger's Algorithm. However, it takes some effort to check all of the details; see Appendix \ref{appendix} for a careful proof.
\end{proof}

Using the above proposition, we can now deduce a relation between Hadamard products and Hilbert functions.  Recall that 
the \emph{Hilbert function} of a variety $V \subseteq \mathbb{P}^n$
is defined by
$$HF_V(i) := \dim_k (R/\mathbb{I}(V))_i 
~~\mbox{for all $i \in \mathbb{N}$} $$
where for any homogeneous ideal $I$, $(R/I)_i$ is the $i$-th degree piece of the $\mathbb{N}$-graded
ring $R/I = \bigoplus_{i \in \mathbb{N}} (R/I)_i.$

With the above result, and the fact that initial ideals yield the same Hilbert functions as the original ideal, we can now relate the Hilbert functions of $V,W,$ and $V \star W$ (when they satisfy the hypotheses of Proposition~\ref{leadingterm}).  Since the Hilbert function computes a number of well-known algebraic invariants of projective varieties, the result on Hilbert functions also gives immediate consequences relating these invariants. We have the following.

\begin{theorem}\label{mainresult3}
Let $V$ and $W$ be binomial varieties
of $\mathbb{P}^n$ that have the same binomial
exponents.  In addition,
suppose that $V$ contains a point 
$p = [p_0:\cdots:p_n]$
with $p_0\cdots p_n \neq 0$, and $W$ also contains a
point $q = [q_0:\cdots:q_n]$ with $q_0\cdots q_n \neq 0$. Let $HF_V, HF_W, HF_{V \star W}$ denote the Hilbert functions associated to $V,W,V \star W$ respectively. 
Then these Hilbert functions are all equal, i.e., 
$$HF_V(i) = HF_W(i) = HF_{V\star W}(i) ~~\mbox{for all 
$i \in \mathbb{N}$}.$$
In particular, $\deg(V) = \deg(W) = \deg(V\star W)$ and
$\dim(V) = \dim(W) = \dim(V\star W).$
\end{theorem}

\begin{proof}
The statements about the degree and dimension follow from the
first statement.  Indeed, the first statement implies
that the Hilbert polynomials of
the rings $R/\mathbb{I}(V), R/\mathbb{I}(W)$ and $R/\mathbb{I}(V \star W)$ are all
equal, and thus the dimension and degrees must all be 
the same since these invariants are encoded into
the Hilbert polynomial (e.g., see 
\cite[Chap. 9, Sec. 3, Theorem 11]{CLO} and 
\cite[Chap. 9, Sec. 4, Exercise 12]{CLO}).

We now prove the statement about the Hilbert functions.
By Theorem \ref{mainresult1}, $\mathbb{I}(V\star W)$ has the same
binomial exponents as $\mathbb{I}(V)$ and $\mathbb{I}(W)$.  Furthermore,
since the Hadamard product of $p$ and $q$ is $p \star q = [p_0q_0:\cdots:p_nq_n]$, and the point
$p\star q$ is contained in $V \star W$, we conclude that $V \star W$ also contains a point with no zero coordinates.  
By Proposition~\ref{leadingterm} applied to the pair $\mathbb{I}(V)$ and $\mathbb{I}(W)$, as well as to the pair $\mathbb{I}(W)$ and $\mathbb{I}(V \star W)$, we can conclude that for any monomial order $<$,
we have
$$LT_<(\mathbb{I}(V)) = LT_<(\mathbb{I}(W)) = LT_<(\mathbb{I}
(V\star W)).$$
Finally, by \cite[Chap. 9, Sec. 3, Theorem 9]{CLO}, the Hilbert
functions  $HF_V, HF_W$ and $HF_{V\star W}$ must all be
 the same since their Hilbert functions are equal
 to the Hilbert functions of 
 $R/LT_<(\mathbb{I}(V)), R/LT_<(\mathbb{I}(W))$, and $R/LT_<(\mathbb{I}(V\star W))$, respectively. This completes the proof. 
\end{proof}


\section{The varieties of Hadamard transformations are binomial}

In the previous sections, we focused on the study of Hadamard products of the form
$V \star W$ where
at least one of the two varieties $V$ and $W$
is assumed to be a binomial variety.  In this 
last section, we shift our perspective, and study Hadamard products of an arbitrary projective variety $V$ with a point $p \in \mathbb{P}^n$. It will then turn out that a certain variety, defined in terms of Hadamard products involving the data from this $V$ and choice of point $p \in \mathbb{P}^n$, is a binomial variety.   

To make the discussion more precise, we need some terminology.  Let $V$ be any projective variety in $\mathbb{P}^n$ and let 
$p = [p_0:\cdots:p_n]\in \mathbb{P}^n$ denote a fixed point in $\mathbb{P}^n$ 
such that $p_0\cdots p_n \neq 0$. We
call $p \star V$ the 
\emph{Hadamard transformation of 
$V$ by $p$}.
Now consider the set
\begin{equation}\label{eq: def Psi V P}
\psi(V,p) := \{ q \in \mathbb{P}^n\, \mid \, q \star V  = p \star V\} \subseteq \mathbb{P}^n 
\end{equation} 
which is the set of points in $\mathbb{P}^n$ which yield the
same Hadamard product with $V$ as 
for $p$.  Note that the set in \eqref{eq: def Psi V P} may not be closed;  however, we can study
the ideal $\mathbb{I}(\psi(V,p))$, i.e., the ideal of polynomials that vanish on $\psi(V,p)$.  Below, we will show that, under some mild hypotheses on $V$,
the ideal $\mathbb{I}(\psi(V,p))$ is generated
by binomials. This in turn implies that the Zariski closure $\overline{\psi(V,p)}$ of $\psi(V,p)$ is a binomial variety.  Morever, it turns out that these binomial equations can be determined from a reduced Gr\"obner basis of $\mathbb{I}(V)$. For a precise statement, see Theorem~\ref{fiberofmap}.

We first require some preparatory lemmas.  As evident in the previous sections, care must be taken when taking the Hadamard transformation of a variety with a point $q$; in particular, we must pay attention to whether the point has some coordinates that are equal to $0$.  The next two lemmas will play
a role in handling this issue.  Recall that if $I$ and $J$ are ideals of a ring $R$, the ideal quotient $I:J$ is defined by 
$I:J = \{f \in R ~|~ fJ \subseteq I\}$.

\begin{lemma}\label{nonzeropoints}
Let $V \subseteq \mathbb{P}^n$ be a nonempty projective 
variety and $p =[p_0:\cdots:p_n] \in \mathbb{P}^n$ with   
$p_0\cdots p_n \neq 0$.
Suppose that  
 $\mathbb{I}(V):\langle x_0x_1\cdots x_n \rangle = \mathbb{I}(V)$.  
If $q =[q_0:\cdots:q_n] \in \psi(V,p)$, then
$q_0\cdots q_n \neq 0$.
\end{lemma}

\begin{proof}
Suppose towards a contradiction that there is a point
$q = [q_0:\cdots:q_n] \in \psi(V,p)$ with at least one zero coordinate. After relabelling, we can assume without loss of generality that $q_0 =0$.
Recall that $q \star V$ is the closure of the set
$$S = \{q \star r ~|~ r \in V, ~~\mbox{$q \star r$ is defined}\} \subseteq \mathbb{P}^n.$$
Since the first coordinate of $q$ is $0$, every point in 
$S$ has $0$ as its first coordinate.  Consequently,
$x_0$ vanishes at all the points of $S$, so $x_0 \in \mathbb{I}(S)$.
Since $q \star V$ is the closure of $S$, we know $\mathbb{V}(\mathbb{I}(S)) = q \star V$. On the other hand, since $x_0 \in \mathbb{I}(S)$, we conclude that $q \star V \subseteq
\mathbb{V}(x_0)$.  But $q \star V  = p \star V$ by the assumption on $q$, which means 
that $p \star V \subseteq \mathbb{V}(x_0)$ also holds. We claim that this implies that 
every point $r = [r_0:\cdots:r_n] \in V$ has $r_0 =0$.  Indeed,
if some point $r \in V$ has $r_0 \neq 0$, then (since $p$ has no zero coordinates, so $p_0 r_0 \neq 0$) we know
$p \star  r$ is defined. By definition of Hadamard products, $p \star r \in p \star V \subseteq \mathbb{V}(x_0)$. However, we have just seen that 
$p \star r$ satisfies $p_0r_0 \neq 0$, so it is not contained in $\mathbb{V}(x_0)$, giving a contradiction.  Hence $r_0=0$.

We have just seen above that $V \subseteq \mathbb{V}(x_0)$. 
This implies that 
$x_0 \in \mathbb{I}(V)$.  Then any multiple of $x_0$ is in $\mathbb{I}(V)$, so in particular $x_0\cdot(x_1\cdots x_n) = 1 \cdot (x_0 x_1 \cdots x_n) \in \mathbb{I}(V)$.  Then by definition of ideal quotients, we have 
$1 \in \mathbb{I}(V):\langle x_0x_1\cdots x_n \rangle$. By the hypothesis on $\mathbb{I}(V)$ we therefore have $1 \in \mathbb{I}(V)$. This implies $V = \mathbb{V}(\mathbb{I}(V)) = \emptyset$, contradicting the hypothesis that $V$ is nonempty.    Thus we conclude that if $q \in \psi(V,p)$, then $q$ cannot have any zero coordinates, as we needed to show. 
\end{proof}

The hypothesis $\mathbb{I}(V):\langle x_0x_1 \cdots x_n\rangle = \mathbb{I}(V)$ played a role in the lemma above.  The next lemma analyzes this situation in more detail.

\begin{lemma}\label{nomonomials}
Let $V \subseteq \mathbb{P}^n$ be a nonempty
projective variety.  If $\mathbb{I}(V):\langle x_0x_1\cdots x_n \rangle
= \mathbb{I}(V)$, then $\mathbb{I}(V)$ contains no monomials.
\end{lemma}

\begin{proof}
Suppose towards a contradiction that $\mathbb{I}(V)$ contains a monomial, say
$x_0^{a_0}\cdots x_n^{a_n}$,  with $a_i \geq 0$.  Let $t
:= \max\{a_0,\ldots,a_n\}$, so that $t-a_i \geq 0$ for all $i$. Then, since $x_0^{a_0} \cdots x_n^{a_n}$ is in $\mathbb{I}(V)$, we have that any multiple is also in $\mathbb{I}(V)$, which in turn implies 
$$(x_0x_1\cdots x_n)^t = x_0^{a_0}x_0^{t-a_0}x_1^{a_1}x_1^{t-a_1}\cdots x_n^{a_n}x_n^{t-a_n} = x_0^{t-a_0} x_1^{t-a_1} \cdots x_n^{t-a_n}(x_0^{a_0} \cdots x_n^{a_n}) \in \mathbb{I}(V).$$
By definition of ideal quotients, this means that $(x_0x_1\cdots x_n)^{t-1} \in \mathbb{I}(V):\langle x_0x_1\cdots x_n
\rangle$. By assumption, we have $\mathbb{I}(V): \langle x_0x_1 \cdots x_n \rangle = \mathbb{I}(V)$, so $(x_0x_1\cdots x_n)^{t-1} \in\mathbb{I}(V)$. But then by the same argument,  we may conclude 
$(x_0x_1\cdots x_n)^{t-2} \in\mathbb{I}(V)$, and so on. We can eventually conclude that $1 \in \mathbb{I}(V)$, so $V = \emptyset$ as in the proof of the previous lemma, contradicting the fact that $V$ is
nonempty.  Thus, $\mathbb{I}(V)$ contains no monomials, as desired. 
\end{proof}

We now come to the main result of this section,
which gives a concrete description of the closure of $\psi(V,p)$ under
some hypotheses on the variety $V$.  In particular, under these hypotheses on $V$, the defining ideal of $\overline{\psi(V,p)}$ turns out to be a binomial ideal. 

\begin{theorem}\label{fiberofmap}
Let $V \subseteq \mathbb{P}^n$ be a
nonempty projective variety.  Suppose that 
$\mathbb{I}(V):\langle x_0x_1 \cdots x_n \rangle = \mathbb{I}(V)$. 
Let $p = [p_0:\cdots:p_n] \in \mathbb{P}^n$ be
a point with $p_0\cdots p_n \neq 0$. 
Let $<$ be a monomial order on $k[x_0,\ldots,x_n]$ and let $\mathcal{G} = \{ f_1, \dots, f_m \}$ denote a reduced Gr\"obner basis for $\mathbb{I}(V)$ with respect to $<$. Assume that every element of the Gr\"obner basis is of degree $\geq 2$. 
For each $i =1,\ldots,m$, write  
$$f_i = X^{\alpha_{1,i}} - a_{2,i}X^{\alpha_{2,i}} - \dots - a_{k_i,i}X^{\alpha_{k_i,i}}
~~\mbox{where $LT(f_i) = X^{\alpha_{1,i}}$}$$
where we assume $X^{\alpha_{k_i,i}} < \cdots < X^{\alpha_{1,i}}$ with respect to $<$ so that $X^{\alpha_{1,i}}$ is the leading term of $f_i$. 
Let
$$J := \langle b_{1,i}X^{\alpha_{1,i}} - b_{\ell,i}X^{\alpha_{\ell,i}} ~|~ 
\mbox{for each $i=1,\ldots,m$ and $1<\ell \leq k_i$} \rangle$$
where the constants $b_{j,i}$ are chosen so 
that they satisfy the equations
$b_{1,i}p^{\alpha_{1,i}} = b_{\ell,i}p^{\alpha_{\ell,i}}$ for all $1 \leq i\leq m$.  
If $\mathbb{I}(\mathbb{V}(J)):\langle x_0x_1\cdots x_n\rangle = J$, 
then $J=\sqrt{J} = \mathbb{I}(\psi(V,p))$.  In particular, under the hypotheses above, the ideal $\mathbb{I}(\psi(V,p))$ is a binomial ideal, and a set of generators of $\mathbb{I}(\psi(V,p))$ can be computed via a reduced Gr\"obner basis of $\mathbb{I}(V)$. 
\end{theorem}

\begin{proof}
We begin by noting that, by the hypothesis on $\mathbb{I}(V)$ and by Lemma 
\ref{nomonomials}, none of the generators $f_i$ in the reduced Gr\"obner basis can be a monomial.  In particular, each $f_i$ has two or more terms,
i.e., $k_i \geq 2$ for all $i$. By the definition of the generators of $J$, we conclude that the generators of $J$ are binomials.  Also,
since the elements of $\mathcal{G}$ are 
a reduced Gr\"obner basis, the leading
coefficient of each $f_i$ is $1$, as 
given in the statement.

Next we claim that there exists a point $q = [q_0:\cdots:q_n] \in\mathbb{V}(J)$ such that $q_0\cdots q_n \neq 0$. To see this, suppose for a contradiction that no such $q$ exists. This implies that for any $q \in \mathbb{V}(J)$, at least one of its coordinates vanishes, i.e., $\mathbb{V}(J) \subseteq \mathbb{V}(\langle x_0 x_1 \cdots x_n \rangle)$. This in turn means that $\mathbb{V}(J) \setminus \mathbb{V}(\langle x_0 x_1 \cdots x_n \rangle) = \emptyset$, and thus 
$$
\mathbb{I}(\mathbb{V}(J) \setminus \mathbb{V}(\langle x_0 x_1 \cdots x_n \rangle)) = \mathbb{I}(\emptyset) = \langle x_0, x_1, \ldots, x_n\rangle
$$ 
since the vanishing ideal of the empty set in $\mathbb{P}^n$ is the irrelevant ideal $\langle x_0, \ldots, x_n\rangle$. For any projective varieties $X$ and $Y$ we have $\mathbb{I}(X \setminus Y) = \mathbb{I}(X):\mathbb{I}(Y)$, so applying that in this case we obtain 
$$
\mathbb{I}(\mathbb{V}(J)):\mathbb{I}(\mathbb{V}(\langle x_0 x_1 \cdots x_n \rangle)) = \mathbb{I}(\mathbb{V}(J)):\langle x_0 x_1 \cdots x_n \rangle = \langle x_0, x_1, \ldots, x_n \rangle
$$ 
where we have used that $\langle x_0 x_1 \cdots x_n \rangle$ is a radical ideal to conclude that $\mathbb{I}(\mathbb{V}(x_0x_1\cdots x_n)) = \langle x_0 \cdots x_n\rangle$.
On the other hand, by hypothesis, the ideal quotient $\mathbb{I}(\mathbb{V}(J)):\langle x_0 x_1 \cdots x_n \rangle$ is equal to $J$, so we conclude $$
J = \langle x_0, x_1, \ldots, x_n \rangle
$$
so in particular it is generated in degree $1$. Note that since $\mathbb{I}(V)$ is a homogeneous ideal, the elements of its reduced Gr\"obner basis $\{f_1, \ldots, f_m\}$ (in the hypothesis of the theorem) are each homogeneous. Moreover, we have assumed that each generator satisfies $\deg(f_i) \geq 2$, which implies that the ideal $J$ is also generated in degree $2$ by its construction. Hence $J$ cannot equal $\langle x_0,\ldots,x_n\rangle$, so we achieve a contradiction. Hence there exists a point $q \in \mathbb{V}(J)$ with $q_0\cdots q_n \neq 0$.

Now let $q$ be such a point, as above.  We know that $q$ satisfies all the relations \begin{equation}\label{eq: relations of q}
b_{1,i}q^{\alpha_{1,i}} = b_{\ell,i}q^{\alpha_{\ell,i}}
~~\mbox{for each $i=1,\ldots,m$ and $1 < \ell \leq k_i$}.
\end{equation}
Since none of the coordinates of $q$ vanish, each monomial $q^{\alpha_{\ell,i}}$ for $1 \leq i \leq m$ and $1 \leq \ell \leq k_i$ appearing in the above binomials is not equal to $0$. Using this we may compute 
the Hadamard transformation of $f_i$ by $q$ as 
\begin{equation}\label{eq: fi star q}
f_i^{\star q} = \frac{1}{{q^{\alpha_{1,i}}}}X^{{\alpha_{1,i}}} - \frac{a_{2,i}}{q^{\alpha_{2,i}}}X^{\alpha_{2,i}} - \dots - \frac{a_{k_i,i}}{{q^{\alpha_{k_i,i}}}}X^{{\alpha_{k_i,i}}}.
\end{equation}
Using the relations~\eqref{eq: relations of q}, we can replace each $q^{\alpha_{\ell,i}}$ with $\frac{b_{1,i}}{b_{\ell,i}}q^{\alpha_{1,i}}$ and so we obtain
$$f_i^{\star q} = \frac{1}{{q^{\alpha_{1,i}}}}X^{{\alpha_{1,i}}} -
\frac{a_{2,i}b_{2,i}}{b_{1,i}q^{\alpha_{1,i}}} X^{\alpha_{2,i}} - \dots - \frac{a_{k_i,i}b_{k_i,i}}{b_{1,i}q^{\alpha_{1,i}}}X^{{\alpha_{k_i,i}}}.$$
 We clear denominators using $b_{1,i}q^{\alpha_{1,i}}$ to find that 
\begin{equation}\label{eq: indep of q} 
b_{1,i}q^{\alpha_{1,i}} \cdot f_i^{\star q}  = b_{1,i}X^{\alpha_{1,i}} - a_{2,i}b_{2,i}X^{\alpha_{2,i}} -\dots - a_{k_i,i}b_{k_i,i}X^{\alpha_{k_i,i}}.
\end{equation}
Observe that the RHS of~\eqref{eq: indep of q} does not depend on the coordinates of $q$.  Moreover, since $b_{1,i}q^{\alpha_1,i}$ is a nonzero constant, we can conclude that the vanishing of $f^{\star q}_i$ is equivalent to the vanishing of the RHS of~\eqref{eq: indep of q}. 
Recall that construction of the coefficients $b_{j,i}$ guarantee that the point $p$ satisfies the same relations~\eqref{eq: relations of q} (with $p$ replacing $q$). Thus, a similar argument as above shows that a nonzero constant multiple of $f_i^{\star p}$ is equal to the RHS of~\eqref{eq: indep of q}, and from this we conclude that $f_i^{\star q}$ is a constant multiple of $f_i^{\star p}$. The above argument did not depend on the index $i$, so this is true for all $i$.  Putting this together with \cite[Theorem 3.5]{BC} (see Theorem~\ref{generatingset}) this means that 
$$\mathbb{I}(q \star V) = \langle f_1^{\star q}, \ldots, f_m^{\star q} \rangle = \langle f_1^{\star p}, \ldots, f_m^{\star p} \rangle  
= \mathbb{I}(p \star V).$$
Thus $q \star V = p \star V$, and consequently,
$q \in \psi(V,p)$ by definition of $\psi(V,p)$.

We are now ready to prove, under our hypotheses, that $\mathbb{I}(\psi(V,p))=J$, which is the main claim of the theorem. 
First, we show that $\mathbb{I}(\psi(V,p)) \subseteq J$. 
We have shown above that for any point $q \in \mathbb{V}(J)$ with $q_0 \cdots q_n \neq 0$, we know that $q \in \psi(V,p)$. In other words, we have shown 
$$ 
\mathbb{V}(J) \setminus \mathbb{V}(x_0x_1\cdots x_n) \subseteq \psi(V,p).
$$
Translating this to ideals, we therefore conclude that 
\begin{equation}\label{eq: 1} \mathbb{I}(\psi(V,p)) \subseteq \mathbb{I}(\mathbb{V}(J) \setminus \mathbb{V}(x_0x_1\cdots x_n)).
\end{equation}
We have already seen above that the RHS of the above inclusion can be described as an ideal quotient, i.e. 
\begin{equation}\label{eq: 2}
\mathbb{I}(\mathbb{V}(J) \setminus \mathbb{V}(x_0x_1\cdots x_n))
= \mathbb{I}(\mathbb{V}(J)):\mathbb{I}(\mathbb{V}(x_0x_1\cdots x_n))
= \mathbb{I}(\mathbb{V}(J)):\langle x_0x_1\cdots x_n \rangle
\end{equation}
where we have used that $\langle x_0 x_1 \cdots x_n \rangle$ is a radical ideal to conclude that $\mathbb{I}(\mathbb{V}(x_0x_1\cdots x_n)) = \langle x_0 \cdots x_n\rangle$. 
By our hypotheses, the last (rightmost) ideal appearing above is equal to $J$. Putting this together with ~\eqref{eq: 1} and~\eqref{eq: 2} gives the desired containment.

Next we show the reverse inclusion $J \subseteq \mathbb{I}(\psi(V,p))$. To see this, suppose $q = [q_0:\cdots:q_n] \in \psi(V,p)$.  By Lemma \ref{nonzeropoints}, we know $q \notin
\mathbb{V}(x_0x_1\cdots x_n)$.  Since $q \star V =
p \star V$, we have $\mathbb{I}(q \star V) = 
\mathbb{I}(p \star V)$.  By assumption, the $\{f_1,\ldots,f_m\}$ form a reduced Gr\"obner basis for $\mathbb{I}(V)$. Hence by Theorem~\ref{reducedgb},
the sets $\{q^{\alpha_{1,1}}f_1^{\star q},\ldots,q^{\alpha_{1,m}}f_m^{\star q}\}$
and $\{p^{\alpha_{1,1}}f_1^{\star p},\ldots,
p^{\alpha_{1,m}}f_m^{\star p}\}$ both form
reduced Gr\"obner bases of $\mathbb{I}(q \star V)=\mathbb{I}(p \star V)$  with
respect to $<$.  Since reduced Gr\"obner bases are unique, and since $LT(q^{\alpha_{1,i}}f_i^{\star q})
= LT(p^{\alpha_{1,i}}f_i^{\star p})$ for
$i = 1,\ldots,m$, we conclude
$q^{\alpha_{1,i}}f_i^{\star q} = p^{\alpha_{1,i}}f_i^{\star p}$ for $i=1,\ldots,m$.

 Using the explicit formula for $f_i^{\star q}$  in~\eqref{eq: fi star q}, (and the similar one for $f_i^{\star p}$) and 
 by comparing the coefficients
 of both sides of $q^{\alpha_{1,i}}f_i^{\star q} = p^{\alpha_{1,i}}f_i^{\star p}$, 
 we have 
 $$
\frac{q^{\alpha_{1,i}}}{q^{\alpha_{\ell,i}}}
= \frac{p^{\alpha_{1,i}}}{p^{\alpha_{\ell,i}}}
 ~~\mbox{
for $1\leq i \leq m$ and $1 < \ell \leq k_i$}.$$
But then because $b_{1,i}p^{\alpha_{1,i}} = b_{\ell,i}p^{\alpha_{\ell,i}}$ for all $1 \leq i\leq m$ and $1 < \ell \leq k_i$,
we have
\begin{eqnarray*}
b_{1,i}q^{\alpha_{1,i}} & = &
b_{1,i}
\left(\frac{q^{\alpha_{\ell,i}}p^{\alpha_{1,i}}}{p^{\alpha_{\ell,i}}} \right) 
= \left(\frac{q^{\alpha_{\ell,i}}(b_{1,i}p^{\alpha_{1,i}})}{p^{\alpha_{\ell,i}}} \right)  
=
\left(\frac{q^{\alpha_{\ell,i}}(b_{1,\ell}p^{\alpha_{\ell,i}})}{p^{\alpha_{\ell,i}}}\right)
= b_{1,\ell}q^{\alpha_{\ell,i}}
\end{eqnarray*}
for all $1 \leq i\leq m$ and $1 < \ell \leq k_i$.
In other words, the point $q$ lies in the vanishing locus of the generators 
of $J$. Thus $q \in \mathbb{V}(J)$. The argument just given shows that $\psi(V,p) \subseteq \mathbb{V}(J)$, so we conclude 
$J \subseteq \mathbb{I}(\mathbb{V}(J)) \subseteq 
\mathbb{I}(\psi(V,p))$.
We have shown both inclusions, so we have shown $J = \mathbb{I}(\psi(V,p))$, as desired. 

Finally, since $J \subseteq \sqrt{J} =
\mathbb{I}(\mathbb{V}(J)))  \subseteq 
\mathbb{I}(\mathbb{V}(J))):\langle x_0x_1
\cdots x_n \rangle = J$, the conclusion 
that $\mathbb{I}(\psi(V,p)) = J = \sqrt{J}$ now follows.
\end{proof}

\begin{example}
Let $V = \mathbb{V}(x^2-xy-yz) \subseteq \mathbb{P}^3$, and 
$p = [1:2:3:4]$. Thus, $\mathbb{I}(V) = \langle x^2-xy-yz \rangle$, 
a principal ideal of $R = k[x,y,z,w]$.  Let $>$ be the lexicographical momomial order 
given by $x>y>z>w$. Since $\mathbb{I}(V)$ is principal, and the leading coefficient 
of $f = x^2-xy-yz$ is 1, we can conclude that $\mathcal{G} = \{f\}$ is a reduced 
Gr\"obner basis for $\mathbb{I}(V)$. 
Furthermore, one can verify by using 
{\it Macaulay2} that $\mathbb{I}(V):\langle xyzw \rangle = \mathbb{I}(V)$, so the above 
theorem applies.  As per Theorem \ref{fiberofmap}, the binomials which generate 
$\mathbb{I}(\psi(V,p))$ are of the form $g_1 = a_1x^2-b_1xy$ and $g_2 = a_2x^2-
b_2yz$. We solve for the coefficients by substituting $x=1$, $y=2$, $z=3$, and 
$w=4$. We have $2b_1 = a_1$ and $6b_2 = a_2$. Therefore, $\mathbb{I}(\psi(V,p)) = 
\langle x^2-(1/2)xy,x^2-(1/6)yz \rangle$.
\end{example}

\noindent
{\bf Acknowledgments.} 
We thank C. Bocci and E. Carlini for answering some of our questions and their suggestions.
Results of this paper were based upon
computer experiments using \emph{Macaulay2}
\cite{M2}, and in particular, the 
{\tt Hadamard} package of Bahmani Jafarloo
\cite{BJ}.
Da Silva was partially supported by an NSERC postdoctoral fellowship.
Harada's research is supported by NSERC Discovery Grant 2019-06567.
Rajchgot's research is supported
by NSERC Discovery Grant 2017-05732.
 Van Tuyl’s research is supported by NSERC Discovery Grant 2019-05412. 
%

\appendix

\section{Proof of Proposition \ref{leadingterm}}\label{appendix}

In this appendix, we include a detailed proof of Proposition \ref{leadingterm}.

\begin{proof}[Proof of Proposition \ref{leadingterm}]
Since $V$ and $W$ are binomial varieties that have the same exponents, we may suppose that 
$\mathbb{I}(V)$ is generated by
$$\mathcal{G} = \{ g_1 = a_1X^{\alpha_1} -b_1X^{\beta_1},
g_2 = a_2X^{\alpha_2} - b_2X^{\beta_2},\ldots, g_s = a_sX^{\alpha_s}-b_sX^{\beta_s} \}$$
and $I(W)$ is generated by
$$\mathcal{H} = \{ h_1 = c_1X^{\alpha_1} -d_1X^{\beta_1},
h_2 = c_2X^{\alpha_2} - d_2X^{\beta_2},\ldots, h_s = c_sX^{\alpha_s}-d_sX^{\beta_s} \}.$$  

By applying  
Buchberger's Algorithm
(see e.g. \cite[Chap. 2, Sec. 7, Theorem 2]{CLO})
to $\mathcal{G}$ (respectively $\mathcal{H}$), we obtain a Gr\"obner basis $\mathcal{G}' = \{g_1, \ldots, g_s, g_{s+1}, \ldots, g_t\} = \mathcal{G} \sqcup \{g_{s+1},\ldots, g_t\}$ (respectively $\mathcal{H}' = \{h_1, \ldots, h_s, h_{s+1}, \ldots, h_p\} = \mathcal{H} \sqcup \{h_{s+1}, \ldots, h_p\}$) of $\mathbb{I}(V)$ (respectively $\mathbb{I}(W)$). Here, the $g_{s+1},\ldots,g_t$ represent the additional elements which may be added to $\mathcal{G}$ by the Buchberger Algorithm in order to produce a Gr\"obner basis, and similarly for $h_{s+1}, \ldots, h_p$. Note that it is possible that $s=t$, in which case $\mathcal{G}$ is already a Gr\"obner basis, and similarly for $h_{s+1},\ldots,h_p$. As discussed before the statement of the proposition, we already know that all $g_i$ and $h_i$ are binomials by \cite[Proposition 1.1]{ES}. We now claim that, in addition, $\mathcal{G}'$ and $\mathcal{H}'$ have the same binomial exponents, i.e., $t=p$ and that the additional elements $g_i$ and $h_i$ are of the form 
\begin{equation}\label{eq: same g h exponents}
g_i = e_iX^{\gamma_i} - f_iX^{\delta_i} \, \textup{ and } \, h_i = l_iX^{\gamma_i}  - m_iX^{\delta_i}
\end{equation} 
for some exponent vectors $\gamma_i, \delta_i$, for all $s+1 \leq i \leq t=p$.

Before proving the above claim, we first show that if this claim is true, then the proposition holds. Indeed, if $g_i$ and $h_i$ share the same monomials for all $1 \leq i \leq t=p$, then we may immediately conclude that for any monomial order, we have $LM_<(g_i)=LM_<(h_i)$, i.e., the leading monomials are equal for all $i$, and hence the leading terms $LT_<(g_i)$ and $LT_<(h_i)$ differ at most by a nonzero scalar multiple. From this it follows that 
$$LT_<(I) = \langle LT_<(g_1),\ldots,LT_<(g_t) \rangle 
= \langle LT_<(h_1),\ldots,LT_<(h_t) \rangle = LT_<(J)$$
where we have used the assumption that $\{g_1,\ldots,g_t\}$ and $\{h_1, \ldots,h_t\}$ are Gr\"obner bases for $I$ and $J$ respectively with respect to $<$ in order to conclude that their leading terms generate the corresponding initial ideals. 

Thus it remains to show that the two Gr\"obner bases $\mathcal{G}'$ and $\mathcal{H}'$ have the same cardinality and also contain binomials
with the same exponents, as expressed in~\eqref{eq: same g h exponents}.  To see this, we need some preliminaries. Recall that the Buchberger Algorithm  \cite[Chap. 2, Sec. 3, Theorem 3]{CLO}) involves computing $S$-polynomials of a set of generators of an ideal, and then applying the division algorithm to the $S$-polynomial; if the remainder upon division is nonzero, then this remainder is added as a generator of the ideal. In our situation, we wish to apply Buchberger's Algorithm to both $\mathcal{G}$ and $\mathcal{H}$ and compare results. Since the generators $g_i$ and $h_i$ (of $\mathcal{G}$ and $\mathcal{H}$ respectively) are ordered, and there are precisely $s$ elements in each generating set, we may assume that we perform Buchberger's Algorithm on $\mathcal{G}$ and $\mathcal{H}$ in precisely the same order. By this we mean that when we compute (for $\mathcal{G}$) the $S$-polynomial of a pair $g_i$ and $g_j$, we may speak of the corresponding step (for $\mathcal{H}$) of computing the $S$-polynomial of the corresponding pair $h_i$ and $h_j$, and the different pairs of indices $(i,j)$ are fed into the two applications of the Buchberger Algorithm (one for $\mathcal{G}$ and one for $\mathcal{H}$) in the same order. With this understanding in place, the goal is now to show that: (1) when applying Buchberger's Algorithm to both $\mathcal{G}$ and $\mathcal{H}$, a step in the algorithm for $\mathcal{G}$ produces an additional element $g_\ell$ to be added to the generating set, if and only if, in the corresponding step for $\mathcal{H}$, the algorithm produces an additional element $h_\ell$ to be added to the generating set. Moreover, we also wish to show that: (2) such additional generators $g_\ell$ and $h_\ell$ must be of the form~\eqref{eq: same g h exponents}, i.e. they must share the same exponents. The statements (1) and (2) would suffice to show the claim.

We argue by induction. More specifically, it is sufficient to show that if we have any two ordered lists of binomials $\{g_1,\ldots,g_r\}$ and $\{h_1,\ldots,h_r\}$ of the same cardinality and sharing the same exponents, then for any two indices $i,j$, the Buchberger Algorithm applied to $g_i,g_j$ and $h_i,h_j$ would either: (i) produce an additional generator for the Gr\"obner basis for both $g_i,g_j$ and $h_i,h_j$, or (ii) not produce an additional generator for either pair. Moreover, we wish to show that in case (i) above, the additional generators, which we may call $g'$ and $h'$, share the same exponents. Applying this argument inductively starting with the base case of $\mathcal{G}$ and $\mathcal{H}$ then suffices to prove (1) and (2) in the above paragraph. 

For the inductive step, we may without loss of generality assume that the pairs under consideration are $g_1 = aX^{\alpha_1} - bX^{\beta_1}, g_2
= cX^{\alpha_2} - dX^{\beta_2}$, and $h_1 = eX^{\alpha_1} -fX^{\beta_1}, h_2 = qX^{\alpha_2} - rX^{\beta_2}$.  Without loss of generality, we may further assume
that $\alpha_i > \beta_i$ for $i=1,2$, with respect to the chosen monomial order $>$. We now compute $S$-polynomials and remainders for both pairs $\{g_1,g_2\}$ and $\{h_1,h_2\}$, following the standard Buchberger method; we will compare the results for the two pairs at every step.    Firstly, the $S$-polynomial of $g_1$ and
$g_2$ is by definition 
\begin{equation}\label{eq: S poly g}
\begin{split} 
S(g_1,g_2) &= \frac{{\rm lcm}(X^{\alpha_1},X^{\alpha_2})}{aX^{\alpha_1}}
(aX^{\alpha_1}-bX^{\beta_1}) - \frac{{\rm lcm}(X^{\alpha_1},X^{\alpha_2})}{cX^{\alpha_2}}
(cX^{\alpha_2}-dX^{\beta_2}) \\
& =  -\frac{b}{a}X^{\beta_1+\alpha'} +\frac{d}{c}X^{\beta_2+\alpha''}\\
\end{split}
\end{equation}
where we define $\alpha', \alpha''$ by $X^{\alpha'} = \frac{{\rm lcm}(X^{\alpha_1},X^{\alpha_2})}{X^{\alpha_1}}$ and
$X^{\alpha''} = \frac{{\rm lcm}(X^{\alpha_1},X^{\alpha_2})}{X^{\alpha_2}}$, and the coefficients appearing in the $S$-polynomial are well-defined since both $a$ and $c$ are nonzero by assumption. 
Similarly we have
\begin{equation}\label{eq: S poly h}
S(h_1,h_2) = -\frac{f}{e}X^{\beta_1+\alpha'} +\frac{r}{q}X^{\beta_2+\alpha''}
\end{equation}
where we know that the exponents $\alpha',\alpha''$ appearing in the computation for $h_1,h_2$ are identical to those appearing for $g_1,g_2$ since $g_i,h_i$ share the same exponents. 

We claim that the two $S$-polynomials in~\eqref{eq: S poly g} and~\eqref{eq: S poly h} are either both equal to $0$, or, they are both binomials and they share the same exponents. To see this, we consider two cases. Suppose first that $\beta_1 + \alpha' \neq \beta_2 + \alpha''$. In this case, both $S(g_1,g_2)$ and $S(h_1,h_2)$ are binomials, and they clearly share the same exponents. Secondly, suppose that $\beta_1 + \alpha' = \beta_2 + \alpha''$. In this case, if $\frac{b}{a} \neq \frac{d}{c}$ (respectively $\frac{f}{e} \neq \frac{r}{q}$), then $S(g_1,g_2)$ (respectively $S(h_1,h_2)$) is a monomial which is by construction contained in $I$ (respectively $J$), and hence there exists a monomial contained in $I$ (respectively a monomial in $J$). However, by hypothesis, we assumed that $V$ and $W$ are both binomial varieties containing a point with no nonzero coordinates. Thus by Lemma~\ref{no_monomial}, there cannot be any monomials in $I=\mathbb{I}(V)$ or in $J=\mathbb{I}(W)$. This implies that we must have $\frac{b}{a}=\frac{d}{c}$ and $\frac{f}{e} = \frac{r}{q}$, which in turn implies that $S(g_1,g_2)=0$ and $S(h_1,h_2)=0$. This proves our claim.

If both $S$-polynomials are $0$, then there are no additional steps for these pairs in the Buchberger Algorithm, so we are done. So suppose now that we are in the case $S(g_1,g_2),S(h_1,h_2) \neq 0$. We claim that when $S(g_1,g_2)$ and respectively $S(h_1,h_2)$ are divided by 
$\mathcal{G} = \{g_1,\ldots,g_r\}$  and respectively
$\mathcal{H} = \{h_1,\ldots,h_r\}$, the current updated generating set, then the remainders -- which we denote $\overline{g}^\mathcal{G}$ and respectively
$\overline{h}^\mathcal{H}$ -- are either both equal to zero, or, they are both 
binomials and they share the same exponents.  

To see this, it is useful to briefly recall the structure of the Division Algorithm procedure as explained in \cite[Chap. 2, Sec. 3, Theorem 3]{CLO}. To compute the remainder of a given polynomial $g$ by $\mathcal{G} = \{g_1, \ldots, g_r\}$ (and similarly for $h$ by $\mathcal{H}$), one starts with a pair $(g,r=0)$ and then applies one step of the Division Algorithm to this pair, thereby producing a new pair $(g',r')$. Here the first entry in any such pair can be thought of as the polynomial ``being divided by $\mathcal{G}$'' and the second entry keeps track of the remainder upon division. The updated pair $(g',r')$ is then fed into the next iteration of the same Division Algorithm, producing the next pair $(g'', r'')$, and so on. When the polynomial being divided becomes $0$, the Division Algorithm terminates, and the current value of the remainder when this occurs is the final remainder, which we denote $\overline{g}^{\mathcal{G}}$. 

With this overall structure in mind, we now proceed with the rest of the argument. Recall that we are assuming that $\mathcal{G}$ and $\mathcal{H}$ consist of binomials and have the same exponents, and so do $S(g_1,g_2)$ and $S(h_1,h_2)$. We first prove that, when applying a step of the Division Algorithm to a binomial $g$ with respect to a set $\mathcal{G}$ consisting also of binomials, then, the output polynomial $g'$ is either a binomial or a monomial. This can be seen by recounting the procedure for producing $g'$. As the base case, suppose $g=S(g_1,g_2) = aX^\alpha - b X^\beta$ and $r=0$ and we are at the initial step in the algorithm. There are two possible cases to consider. Case (1) is when there exists $g_i = a_i X^{\alpha_i} - b_i X^{\beta_i}$ such that $LT_<(g_i)$ divides $LT_<(g)$. Case (2) is when no such $g_i$ exists. In case (1), the algorithm replaces $g=S(g_1,g_2)$ with 
$$
g - \frac{LT_<(g)}{LT_<(g_i)} g_i = aX^{\alpha} - bX^{\beta} - \frac{LT_<(g)}{LT_<(g_i)}(
a_iX^{\alpha_i}-b_iX^{\beta_i}) = 
-bX^{\beta} + \frac{ab_i}{a_i}X^{\beta_i+\alpha'}$$
where $X^{\alpha'} = \frac{X^\alpha}{X^{\alpha_i}}.$
Since $h = S(h_1,h_2)= dX^{\alpha} - eX^{\beta}$ has the 
same exponents as $g$, and because $h_i =
d_iX^{\alpha_i} - e_iX^{\beta_i} \in \mathcal{H}$ also satisfies $LT_<(h_i)|LT_<(h)$, when we run the 
division algorithm with input $h=S(h_1,h_2)$ and $\mathcal{H}$, it
will replace $h$ with 
$$-eX^{\beta} + \frac{de_i}{d_i}X^{\beta_i+\alpha'}.$$

By an argument similar to that given above, by Lemma~\ref{no_monomial} we know that if $\beta=\beta_i + \alpha'$ then $g'=h'=0$ and the Division Algorithm terminates, returning as $\overline{g}^{\mathcal{G}}$ and $\overline{h}^{\mathcal{H}}$ the current value of the remainders which, at this initial step, is $r=0$. Thus in this case, the remainders are both $0$ as claimed. If $\beta \neq \beta_i + \alpha'$ then both $g'$ and $h'$ are binomials and they have the same exponents as claimed. Now suppose we are in case (2). Note this also means that no leading term of an element in $\mathcal{H}$ divides $h = cX^{\alpha}-dX^{\beta}$, since $h$ and $g$, and $\mathcal{G}$ and $\mathcal{H}$, respectively share the same exponents. In this setting, the algorithm for $g$ sets the value of the remainder $r$ to  $aX^{\alpha}$ and $g$ is replaced with $g' = -bX^{\beta}$. Similarly, the algorithm for $h$ sets the remainder $r$ to be $cX^{\alpha}$ and $h$ is replaced with
$h'= -dX^{\beta}$. Thus $g'$ and $h'$ in this case are both monomials sharing the same exponent. 

Next observe that the argument given above for $g=S(g_1,g_2)$ and $h=S(h_1,h_2)$ is valid for any input polynomials $g,h$ which are binomials that share the same exponents and the remainder terms $r$ are both equal to $0$. This also shows that as long as the Division Algorithm remains in Case (1), the output polynomials $g',h'$ will either be both binomials that have the same exponents, or, they are both equal to $0$. Moreover, if the algorithm terminates while only encountering case (1), then both remainders $\overline{g}^{\mathcal{G}}$ and $\overline{h}^{\mathcal{H}}$ are $0$. 

We now claim that when a step in the Division Algorithm is applied to a (nonzero) monomial input, so $g=aX^\alpha$, then for all remaining steps in the Division Algorithm, the output $g'$ is either also a monomial, or, equal to $0$. (Informally, we can think of this as: ``once a monomial, always a monomial -- until the algorithm ends''.) To see this, observe again that when the input polynomials are $g=-bX^\beta$ and $h=-dX^\beta$, nonzero monomials of the same exponent, we again have two cases: Case (1) is when there exists $g_i=a_i X^{\alpha_i} - b_i X^{\beta_i}$ such that $LT_<(g_i) \vert LT_<(g)$. As above, this occurs if and only if $LT_<(h_i) \vert LT_<(h)$. If there is such an element,
then $g$ is replaced by 
$$-bX^{\beta} - \frac{-bX^{\beta}}{a_iX^{\alpha_i}}(a_iX^{\alpha_i}
-b_iX^{\beta_i}) = -\frac{bb_i}{a_i}X^{\beta_i+\beta'}$$
where $\beta' = \frac{X^\beta}{X^{\alpha_i}}.$ Note that $g'$ is again monomial, as claimed. An analogous computation shows that $h$ is replaced by $-\frac{dd_i}{c_i}X^{\beta_i+\beta'}$, which is again monomial and has the same exponent as $g'$. In fact, more is true. We saw above that the input polynomial becomes a monomial when no $LT_<(g_i)$ (respectively $LT_<(h_i)$) divides $LT_<(g)$ (respectively $LT_<(h)$), and in this case, the remainder terms are updated to nonzero monomials. If in the next steps with inputs $g=aX^\alpha, h=bX^\alpha$, we are in Case (1), then the remainders remain unchanged, and hence are still monomials. 

Now we consider case (2), when there is no $g_i$ such that $LT_<(g_i) \vert LT(g)$ and hence similarly for $h_i$ and $h$. In such a case, the Division Algorithm dictates that the leading terms $LT_<(g)$ and $LT_<(h)$ be added to their respective remainder terms and we set the new polynomials to be $g' = g-LT_<(g)$ and $h'=h-LT_<(h)$. However, since $g$ and $h$ are monomials, we conclude $g=LT_<(g)$ and $h=LT_<(h)$ and so $g'=h'=0$. Thus the algorithm terminates at the first instance when the $g,h$ are in case (2). 

What we have now seen is that, if the Division Algorithm is applied to $S(g_1,g_2)$ and $S(h_1,h_2)$, then either both remainders $\overline{g}^{\mathcal{G}}$ and $\overline{h}^{\mathcal{H}}$ are equal to $0$ or, if Case (2) is ever encountered and the remainder term becomes nonzero, then the remainder term $r$ changes exactly twice: once when a binomial $g$ is changed to a monomial $g'$, and once more when a monomial input $g$ encounters case (2) -- and similarly for $h$. Note that this final remainder is indeed binomial and not $0$ (which could occur if there was a cancellation), because the construction of the Division Algorithm implies that the second monomial which gets added to the remainder term is strictly less (in the monomial order) than the first. Thus, the final remainder is a binomial as claimed. Moreover, in both situations, when $r$ changes by a monomial, we have seen that the exponents of the monomials appearing for both $g$ and $h$ are the same, thus implying that the two remainders $\overline{g}^{\mathcal{G}}$ and $\overline{h}^{\mathcal{H}}$ share the same exponents. This completes the proof. 
\end{proof}

\end{document}